\documentclass[11pt,twoside]{amsart}
\usepackage{amsmath,amssymb,amsthm,stmaryrd}
\usepackage{mathrsfs}
\usepackage{hyperref}

\makeatletter
\@namedef{subjclassname@2020}{%
  \textup{2020} Mathematics Subject Classification}
\makeatother

\numberwithin{equation}{section}

\newtheorem{lemma}[equation]{Lemma}
\newtheorem{theorem}[equation]{Theorem}
\newtheorem{proposition}[equation]{Proposition}

\newcommand{\R}{{\mathbb R}}
\newcommand{\N}{{\mathbb N}}
\newcommand{\Z}{{\mathbb Z}}

\newcommand{\bI}{{\mathbf I}}
\newcommand{\bZ}{{\mathbf Z}}
\newcommand{\M}{{\mathbf M}}
\newcommand{\W}{{\mathbf W}}
  
\newcommand{\cD}{{\mathscr D}}
\newcommand{\cF}{{\mathscr F}}
\newcommand{\cG}{{\mathscr G}}
\newcommand{\cP}{{\mathscr P}}
\newcommand{\cS}{{\mathscr S}}
\newcommand{\cW}{{\mathscr W}}
\newcommand{\cX}{{\mathscr X}}

\newcommand{\del}{\delta}
\newcommand{\Del}{\Delta}
\newcommand{\eps}{\epsilon}
\newcommand{\gam}{\gamma}
\newcommand{\Gam}{\Gamma}

\newcommand{\lam}{\lambda}

\newcommand{\Om}{\Omega}
\renewcommand{\rho}{\varrho}

\newcommand{\B}[2]{B_{#1}(#2)}   

\newcommand{\olA}{\,\overline{\!A}}

\newcommand{\EII}{{\rm EII}}
\newcommand{\LII}{{\rm LII}}
\newcommand{\SII}{{\rm SII}}
\newcommand{\CI}{{\rm CI}}
\newcommand{\FR}{{\rm FR}}
\renewcommand{\SS}{{\rm SS}}
\newcommand{\AR}{{\rm AR}}
\newcommand{\ML}{{\rm ML}}

\newcommand{\bb}[1]{\llbracket #1\rrbracket} 
\renewcommand{\d}{\partial}
\newcommand{\on}{\mathop{\mbox{\rule{0.1ex}{1.2ex}\rule{1.1ex}{0.1ex}}}}

\newcommand{\diam}{\operatorname{diam}}
\newcommand{\D}{\Theta}
\newcommand{\cs}{\text{\rm c}} 
\newcommand{\spt}{\operatorname{spt}}
\newcommand{\CAT}{\operatorname{CAT}}
\newcommand{\FillVol}{\operatorname{FillVol}}
\newcommand{\loc}{\text{\rm loc}}
\newcommand{\sub}{\subset}
\newcommand{\sm}{\setminus}
\newcommand{\es}{\emptyset}
\newcommand{\Lip}{\operatorname{Lip}}

\begin{document}

\title{Characterizations of higher rank hyperbolicity}
\author{Tommaso Goldhirsch}
\address{Department of Mathematics\\ETH Z\"urich\\R\"amistrasse 101
  \\8092 Z\"urich\\Switzerland}
\email{tommaso.goldhirsch@math.ethz.ch}
\author{Urs Lang}
\address{Department of Mathematics\\ETH Z\"urich\\R\"amistrasse 101
  \\8092 Z\"urich\\Switzerland}
\email{lang@math.ethz.ch}
\date{August 1, 2023}
\thanks{Research supported by Swiss National Science Foundation Grant 197090.}

\begin{abstract}  
In analogy to the various characterizations of Gromov hyperbolicity,  
we present a list of six mutually equivalent higher rank conditions
for metric spaces satisfying some assumption reminiscent of global
non-positive curvature.
\end{abstract}

\maketitle


\section{Introduction}

The concept of Gromov hyperbolicity manifests itself in many different ways.
With only mild assumptions on the underlying metric space,
the spectrum of equivalent properties includes various thin triangle
conditions, the stability of quasi-geodesics (the Morse lemma),
a linear isoperimetric filling inequality for closed curves,
and a sub-quadratic isoperimetric inequality
\cite{Bon, Bow1, Bow2, BriH, BuyS, Coo+, Ghy+, Gro-HG, Pap, Sho+, Vai, Wen-IC}.
We present a similar list of six equivalent properties in the context of
generalized non-positive curvature and higher asymptotic rank.
This complements the results in~\cite{Wen-AR} and in the recent
paper~\cite{KleL}. We give a largely self-contained proof, providing
some improvements and simplifications for the known part.

For an informal statement of the main result, let us focus on the
special case that $X$ is a proper $\CAT(0)$ or Busemann convex space.
In passing from Gromov hyperbolicity to rank $n \ge 2$, the
role of closed curves and quasi-geodesics is transferred to $n$-cycles and
$n$-chains satisfying a suitable quasi-minimality condition, respectively.
For the moment, the reader is invited to think of the chain complex of
Lipschitz singular chains with integer coefficients in $X$. Some of the
statements below involve a uniform polynomial mass bound of degree $n$ in
large balls, and we shall thus speak of $n$-chains with
{\em controlled density}. This condition holds automatically
for Lipschitz quasi-geodesics if $n = 1$ or, more generally, for Lipschitz
quasi-isometric embeddings of domains in $\R^n$ into $X$.
We show that the following are equivalent:
\begin{itemize}
\item
  the {\em asymptotic rank\/} of $X$ being at most $n$ (see below
  for the definition);
\item
  a {\em sub-Euclidean isoperimetric inequality\/} for $n$-cycles,
  corresponding to a sub-quadratic inequality in the case $n = 1$;
\item
  a {\em linear isoperimetric inequality\/} for $n$-cycles with controlled
  density;
\item
  a version of the {\em Morse lemma\/} implying in particular a bound on the
  Hausdorff distance between (the supports of) two quasi-minimizing $n$-chains
  with controlled density and equal boundary;
\item a {\em slim $(n+1)$-simplex property\/} analogous to the slimness of
  quasi-geodesic triangles in geodesic Gromov hyperbolic spaces;
\item a bound on the {\em filling radius\/} of $n$-cycles with controlled
  density.  
\end{itemize}

We now proceed to the details. The actual setup is as in~\cite{KleL}.
Suppose that $X = (X,d)$ is a proper metric space, that is, closed bounded
subsets are compact. We use the chain complex $\bI_{*,\cs}(X)$ of metric
integral currents with compact support, which comprises the singular Lipschitz
chains but is more versatile and has suitable compactness properties.
The relevant prerequisites from the theory of metric currents will be
discussed in Sect.~\ref{sect:prelim}.  
For $n \ge 1$, we say that $X$ satisfies {\em condition $(\CI_n)$} if there
is a constant $c$ such that any two points $x,y$ in $X$ can be joined by
a curve of length $\le c\,d(x,y)$, and for $k = 1,\ldots,n$, every $k$-cycle
$R \in \bI_{k,\cs}(X)$ in some $r$-ball is the boundary of an
$S \in \bI_{k+1,\cs}(X)$ with mass
\[
  \M(S) \le c\,r\,\M(R).
\]
The cone inequalities~$(\CI_n)$ hold in particular, for all $n$,
if $X$ is a $\CAT(0)$ space or a space with a conical geodesic
bicombing~\cite{DesL, Wen-EII}.
We remark that every hyperbolic group acts geometrically on a proper
polyhedral complex with such a bicombing~\cite{Lan-IH}, and further classes
of groups with this property are discussed in~\cite{Cha+, Hae, HuaO, OsaV}.
Moreover, condition~$(\CI_n)$ holds if $X$ is an $n$-connected Riemannian
manifold with a geometric action of a (quasi-geodesically) combable group;
compare Theorem~10.3.5 in~\cite{Eps+}.

The {\em asymptotic rank\/} of $X$ is the supremum of all $k \ge 0$ for
which there exist a sequence $0 < r_i \to \infty$ and subsets $Y_i \sub X$
such that the rescaled sets $(Y_i,\frac{1}{r_i}d)$ converge in the
Gromov--Hausdorff topology to the unit ball in some $k$-dimensional normed
space. This is a quasi-isometry invariant, and if $X$ is a geodesic metric
space satisfying $(\CI_1)$, then the asymptotic rank is at most $1$ if and
only if $X$ is Gromov hyperbolic~\cite{Wen-AR}.
If $X$ is a cocompact $\CAT(0)$ space or a cocompact space with a
conical geodesic bicombing, then the asymptotic
rank equals the maximal dimension of an isometrically embedded Euclidean
or normed space, respectively~\cite{Des, Kle}.

Let $S \in \bI_{n,\cs}(X)$. For constants $C \ge 1$ and $a \ge 0$, we say that
$S$ has {\em $(C,a)$-controlled density\/} if for all $x  \in X$ and $r > a$,
the piece of $S$ in the closed $r$-ball at $x$ has mass at most $Cr^n$, or
\[
  \D_{x,r}(S) := \frac{1}{r^n}\M(S \on \B{x}{r}) \le C.
\]
For almost every $r > 0$, the boundary $\d(S \on \B{x}{r})$ has finite mass
and $S \on \B{x}{r}$ is itself an element of $\bI_{n,\cs}(X)$
(see again Sect.~\ref{sect:prelim}).
Suppose that $Y \sub X$ is a closed set containing the support
$\spt(\d S)$ of $\d S$. For $Q \ge 1$ and $a \ge 0$, we say that $S$ is
{\em $(Q,a)$-quasi-minimizing mod~$Y$} if for every point $x \in \spt(S)$
at a distance $b > a$ from $Y$, the inequality
\[
  \M(S \on \B{x}{r}) \le Q\,\M(T)
\]
holds for almost all $r \in (a,b)$ and all $T \in \bI_{n,\cs}(X)$
with $\d T = \d(S \on \B{x}{r})$. 
A current $S \in \bI_{n,\cs}(X)$ is called a {\em $(Q,a)$-quasi-minimizer\/}
if $S$ is $(Q,a)$-quasi-minimizing mod~$\spt(\d S)$.
As for the analogy with quasi-geodesics, one can easily check that
for a Lipschitz quasi-isometric embedding $\gam \colon [0,l] \to X$ of an
interval, the associated current $\gam_\#\bb{0,l} \in \bI_{1,\cs}(X)$
is a quasi-minimizer with controlled density (compare the case $n = 1$ of
Proposition~\ref{prop:lip-qflats}). For $n > 1$, quasi-minimizers offer
more flexibility than quasiflats.
  
We can now state the main result of this paper. Detailed comments
and references are given below.

\begin{theorem} \label{thm:main}
Suppose that $X$ is a proper metric space satisfying condition~{\rm (CI$_n$)}
for some $n \ge 1$. Then the following six properties are equivalent:
\begin{enumerate}
\item[$(\AR_n)$]
  {\em (asymptotic rank)} the asymptotic rank of $X$ is at most $n$;
\item[$(\SII_n)$]
  {\em (sub-Euclidean isoperimetric inequality)} for all $\eps > 0$ there
  is a constant $M_0 > 0$ such that every cycle $Z \in \bI_{n,\cs}(X)$ is
  the boundary of a $V \in \bI_{n+1,\cs}(X)$ with mass\/
  $\M(V) < \eps \max\{M_0,\M(Z)\}^{(n+1)/n}$;
\item[$(\LII_n)$]
  {\em (linear isoperimetric inequality)} there is a constant $\nu > 0$,
  and for all $C > 0$ there is a $\lam > 0$, such that every cycle
  $Z \in \bI_{n,\cs}(X)$ with $(C,a)$-controlled density bounds a
  $V \in \bI_{n+1,\cs}(X)$ with\/ $\M(V) \le \max\{\lam,\nu a\}\,\M(Z)$;
\item[$(\ML_n)$]
  {\em (Morse lemma)} for all $C > 0$ and $Q \ge 1$ there is a constant
  $l \ge 0$ such that if $Z \in \bI_{n,\cs}(X)$ is a cycle with
  $(C,a)$-controlled density and $Y \sub X$ is a closed set such that $Z$ is
  $(Q,a)$-quasi-minimizing mod~$Y$, then $\spt(Z)$ is within distance
  at most\/ $\max\{l,4a\}$ from~$Y$;
\item[$(\SS_n)$]
  {\em (slim simplices)} for all $L \ge 1$ there is a constant $D \ge 0$
  such that if $\Del$ is a Euclidean $(n+1)$-simplex and
  $f \colon \d\Del \to X$ is a map whose restriction to each facet of $\Del$
  is an $(L,a)$-quasi-isometric embedding, then the image of every facet is
  within distance at most $D(1 + a)$ of the union of the images of the
  remaining ones;
\item[$(\FR_n)$]
  {\em (filling radius)} for all $C > 0$ there is a constant $h > 0$ such
  that every cycle $Z \in \bI_{n,\cs}(X)$ with $(C,a)$-controlled density
  bounds a $V \in \bI_{n+1,\cs}(X)$ whose support is within distance at most\/
  $\max\{h,a\}$ from $\spt(Z)$.
\end{enumerate}
\end{theorem}

Note that for $n = 1$, $(\SII_n)$ corresponds to a sub-quadratic inequality.
The equivalence of $(\AR_n)$ and $(\SII_n)$ was established in~\cite{Wen-AR}
in a more general setup for complete metric spaces. The proof of the forward
implication used an elaborate thick-thin decomposition for integral cycles
from~\cite{Wen-CT} and also the non-trivial fact that a weakly convergent
sequence of cycles converges with respect to the filling volume~\cite{Wen-FC}.
We review the entire argument. Employing an elegant new variational result
from~\cite{HuaKS-1} and introducing a more quantitative approach for the
weak convergence of cycles, we reduce the overall complexity substantially.
In fact, we prove the sub-Euclidean isoperimetric inequality first in a
somewhat restricted form (Theorem~\ref{thm:sii1}, compare Theorem~4.4
in~\cite{KleL}) and then deduce $(\LII_n)$ and $(\SII_n)$.

The implication $(\AR_n) \Rightarrow (\LII_n)$ pertains to a long-standing
open problem. In symmetric spaces of non-compact type or homogeneous
Hadamard manifolds of rank $\le n$, a linear isoperimetric inequality holds
for all $n$-cycles (see p.~105 in~\cite{Gro-AI}, \cite{Leu}, and~\cite{Isl}).
It is still unkown whether this generalizes, for instance, to cocompact
Hadamard manifolds or $\CAT(0)$ spaces of (asymptotic) rank $\le n$.
In our statement, the isoperimetric constant depends on the density bound,
so Theorem~\ref{thm:main} does not resolve this question.
Nevertheless, $(\LII_n)$ turns out to be equivalent to the remaining
properties. Examples of cycles with controlled density include
cycles lying near the union of finitely many quasiflats (see in particular
the proof of Theorem~\ref{thm:slim}).

The proofs of Theorem~5.1 and Theorem~5.2 in~\cite{KleL} show that
$(\SII_n) \Rightarrow (\ML_n) \Rightarrow (\SS_n)$, except for a less
explicit distance bound in the slim simplex property. The second step
involves an approximation result for quasiflats by quasi-minimizers.
We go through the argument in detail, keeping track of the dependence of
constants, and thus showing that the bound is linear in the
coarseness parameter $a$. This fact is used in the proof of
the backward implication $(\SS_n) \Rightarrow (\AR_n)$.

The statement of $(\ML_n)$ differs formally from the usual stability
assertion for quasi-geodesics, but is versatile.
If $S \in \bI_{n,c}(X)$ with $\spt(\d S) \sub Y$ is
quasi-minimizing mod $Y$, then the extra assumption we need in order to
conclude that $S$ is confined to a bounded neighborhood of $Y$
is that $S$ can be closed up to a cycle $Z = S - S'$
with controlled density and with $\spt(S') \sub Y$. Note that there
is no (quasi-)minimality assumption on $S'$; the density bound suffices.
However, if $S_1,S_2 \in \bI_{n,\cs}(X)$ are two $(Q,a)$-quasi-minimizers with
$\d S _1 = \d S_2$, each with $(C,a)$-controlled density, then $S_1$ is
$(Q,a)$-quasi-minimizing mod $\spt(S_2)$ and vice-versa, so $(\ML_n)$ implies
that the Hausdorff distance between the supports is bounded by $\max\{l,4a\}$
for $l = l(2C,Q)$.

The last assertion of Theorem~\ref{thm:main} is yet another way of
expressing that $n$-cycles with controlled density (such as geodesic
triangles if $n = 1$) are thin. We use an iterative application of the
sub-Euclidean isoperimetric inequality to show that the conclusion of
$(\FR_n)$ holds for every mass minimizing $V$ with $\d V = Z$.
In~\cite{Wen-AR}, the filling radius was used to prove that
$(\SII_n) \Rightarrow (\AR_n)$. Similarly, $(\FR_n) \Rightarrow (\AR_n)$.

The paper is organized as follows. In Sect.~\ref{sect:prelim} we recall the
definition of metric currents and collect some basic results.
In Sect.~\ref{sect:var} we first review an approximation result for cycles
from~\cite{HuaKS-1} and then use this to give a short proof of a quantitative
version of a result from~\cite{Wen-FC}, showing that cycles with
bounded mass and sufficiently small uniformly bounded density at some
fixed scale have small filling volume. We use this further in
Sect.~\ref{sect:conv} to discuss the convergence of cycles. 
Sect.~\ref{sect:ii} is then devoted to isoperimetric inequalities and
shows in particular that
\[
  (\AR_n) \Rightarrow (\LII_n) \Leftrightarrow (\SII_n).
\]
In Sect.~\ref{sect:qmin} we prove more explicit versions of two propositions
from~\cite{KleL} relating quasiflats and quasi-minimizers.
The concluding Sect.~\ref{sect:slim} then shows in particular that
\[
(\SII_n) \Rightarrow (\ML_n) \Rightarrow (\SS_n) \Rightarrow (\AR_n)
\quad \text{and} \quad (\SII_n)  \Rightarrow (\FR_n) \Rightarrow (\AR_n).
\]
In fact, we prove all implications (and hence Theorem~\ref{thm:main})
in a stronger form, for any class of proper metric spaces satisfying
the respective assumptions uniformly, and with constants depending only
on the data involved and on the class, rather than on individual members
(see Sect.~\ref{sect:ii} and Sect.~\ref{sect:slim}).

What is missing from the list in Theorem~\ref{thm:main} is a rank $n$ analog
of Gromov's quadruple definition of $\delta$-hyperbolicity
(\cite{Gro-HG}, p.~89).
A $2(n+1)$-point condition of this type is investigated in~\cite{JorL}.


\section{Preliminaries} \label{sect:prelim}

Currents with finite mass in complete metric spaces were introduced by
Ambrosio and Kirchheim in~\cite{AmbK}. Here, for consistency with~\cite{KleL},
we will use the local theory described in~\cite{Lan-LC}. For the class of
integral currents with compact support, the principal objects in this paper,
the formal difference between the two approaches is marginal. Assuming the
underlying metric space to be proper, we will make frequent use of the
existence of area minimizing integral currents filling a given cycle
(Theorem~\ref{thm:plateau}). Without this assumption, as an additional twist,
one could still work with almost minimal currents instead (see, for example,
Lemma~3.4 in~\cite{Wen-EII}). In particular, Theorem~\ref{thm:main} and most
results in the paper hold more generally for complete metric spaces and
Ambrosio--Kirchheim integral currents.

\subsection*{Currents}
An integral $n$-current may roughly be thought of as an oriented
$n$-dimensional Lipschitz surface equipped with a summable integer density
function. Formally though, $n$-currents are defined as functionals;
on compactly supported differential $n$-forms in the classical case
(going back to de Rham), and on suitable $(n+1)$-tuples of
real-valued locally Lipschitz functions for non-smooth ambient spaces.
The relating principle (originally proposed by De Giorgi) is that the tuple
$(f_0,\ldots,f_n)$, say if the $f_i$ are smooth functions on $\R^N$,
represents the form $f_0\,df_1 \wedge \ldots \wedge df_n$.

Let $X = (X,d)$ be a proper metric space. For $n \ge 0$, we let $\cD^n(X)$
denote the set of all $(n+1)$-tuples $(f_0,\ldots,f_n)$ of Lipschitz
functions $f_i \colon X \to \R$ such that $f_0$ has compact support
$\spt(f_0)$ (in~\cite{Lan-LC}, $f_1,\ldots,f_n$ are merely locally Lipschitz,
but the following definition is equivalent). An
{\em $n$-dimensional current\/} $S$ in $X$ is a function
$S \colon \cD^n(X) \to \R$ satisfying the following three conditions:
\begin{enumerate}
\item
$S$ is $(n+1)$-linear;
\item 
$S(f_{0,k},\ldots,f_{n,k}) \to S(f_0,\ldots,f_n)$
whenever $f_{i,k} \to f_i$ pointwise on $X$ with uniformly bounded
Lipschitz constants ($i = 0,\dots,n$) and with $\bigcup_k\spt(f_{0,k}) \sub K$
for some compact set $K \sub X$; 
\item
$S(f_0,\ldots,f_n) = 0$ whenever one of the functions
$f_1,\ldots,f_n$ is constant on a neighborhood of $\spt(f_0)$.
\end{enumerate}
It follows from these axioms that $S$ is alternating in the
last $n$ arguments. The vector space of all $n$-dimensional currents 
in $X$ is denoted $\cD_n(X)$. Every function $w \in L^1_\loc(\R^n)$ induces
a current $\bb{w} \in \cD_n(\R^n)$ defined by
\[
\bb{w}(f_0,\dots,f_n) 
:= \int w f_0\det\bigl[\d_jf_i\bigr]_{i,j = 1}^n \,dx;
\]
note that the partial derivatives $\d_jf_i$ exist almost everywhere by
Rade\-ma\-cher's theorem. For a Borel set $W \sub \R^n$ we put
$\bb{W} := \bb{\chi_W}$, where $\chi_W$ denotes the characteristic function.
(See Sect.~2 in~\cite{Lan-LC} for details.)

\subsection*{Support, push-forward, and boundary}
Let $S \in \cD_n(X)$. There exists a smallest closed subset of $X$,
the {\em support\/} $\spt(S)$ of $S$, such that the value $S(f_0,\ldots,f_n)$
depends only on the restrictions of $f_0,\dots,f_n$ to this set.
Thus, for any closed set $D \sub X$ containing $\spt(S)$,
$S$ induces a current in $\cD_n(D)$, still denoted by $S$.
For a proper Lipschitz map $\phi \colon D \to Y$ into another proper
metric space $Y$, the {\em push-forward\/} $\phi_\#S \in \cD_n(Y)$ is the
current with support in $\phi(\spt(S))$ defined by
\[
(\phi_\#S)(f_0,\ldots,f_n) := S(f_0 \circ \phi,\ldots,f_n \circ \phi)
\]
for all $(f_0,\ldots,f_n) \in \cD^n(Y)$. In the simplest case,
if $\bb{a,b} := \bb{[a,b]}$ is the current in $\cD_1(\R)$ (or $\cD_1([a,b])$)
associated with an interval, and if $\gam \colon [a,b] \to X$ is a Lipschitz
curve, then
\[
  \gam_\#\bb{a,b}(f_0, f_1) = \bb{a,b}(f_0 \circ \gam, f_1 \circ \gam) =
  \int_a^b (f_0 \circ \gam)(f_1 \circ \gam)' \,ds
\]
for all $(f_0,f_1) \in \cD^1(X)$. Similarly, every singular Lipschitz
$n$-chain in $X$ defines an element of $\cD_n(X)$; in fact, of $\bI_{n,\cs}(X)$
(see below for the definition, and~\cite{BasWY, Gol} for some reverse
approximation results). 

If $S \in \cD_n(X)$ and $n \ge 1$, then the {\em boundary}
$\d S \in \cD_{n-1}(X)$ is defined by
\[
(\d S)(f_0,\dots,f_{n-1}) := S(\tau,f_0,\dots,f_{n-1})
\]
for all $(f_0,\ldots,f_{n-1}) \in \cD^{n-1}(X)$ and for any $\tau \in \cD^0(X)$
such that $\tau \equiv 1$ in a neighborhood of $\spt(f_0)$.
It follows from~(1) and~(3) that $\d S$ is well-defined and that
$\d \circ \d = 0$. Furthermore, $\spt(\d S) \sub \spt(S)$, and
$\phi_\#(\d S) = \d(\phi_\# S)$ for $\phi \colon D \to Y$ as above.
In the example of a Lipschitz curve,
$\d(\gam_\#\bb{a,b})(f_0) = f_0(\gam(b)) - f_0(\gam(a))$ by the fundamental
theorem of calculus. (See Sect.~3 in \cite{Lan-LC}.)

\subsection*{Mass}
Let $S \in \cD_n(X)$. For an open set $U \sub X$, 
the {\em mass} $\|S\|(U) \in [0,\infty]$ of $S$ in $U$ is defined as
the supremum of $\sum_k S(f_{0,k},\ldots,f_{n,k})$
over all finite families of tuples $(f_{0,k},\ldots,f_{n,k}) \in \cD^n(X)$
such that $\bigcup_k \spt(f_{0,k}) \sub U$, $\sum_k |f_{0,k}| \le 1$, and
$f_{1,k},\ldots,f_{n,k}$ are $1$-Lipschitz.
This extends to a regular Borel measure $\|S\|$ on $X$ with
$\spt(\|S\|) = \spt(S)$, and $\M(S) := \|S\|(X)$ denotes the {\em total mass}.
For Borel sets $W,A \sub \R^n$, $\|\bb{W}\|(A)$ equals the Lebesgue measure
of $W \cap A$. For $S,T \in \cD_n(X)$,
\[
\|S + T\| \le \|S\| + \|T\|.
\]
If the measure $\|S\|$ is locally finite, then
\[
|S(f_0,\ldots,f_n)| \le \prod_{i=1}^n \Lip(f_i) \int_X |f_0| \,d\|S\|
\]
for all $(f_0,\ldots,f_n) \in \cD^n(X)$, where $\Lip(f_i)$ denotes the
Lipschitz constant. As a consequence, $S$ extends to
tuples whose first entry is merely a bounded Borel function
with compact support, and if $u \colon X \to \R$ is any bounded Borel
function, one can define the {\em restriction} $S \on u \in \cD_n(X)$ by
\[
(S \on u)(f_0,\dots,f_n) := S(uf_0,f_1,\ldots,f_n)
\]
for all $(f_0,\ldots,f_n) \in \cD^n(X)$. For a Borel set $A \sub X$,
$S \on A := S \on \chi_A$.
The measure $\|S \on A\|$ agrees with the restriction of $\|S\|$ to $A$.
If $\phi \colon D \to Y$ is as above, and $B \sub Y$ is a Borel set, then
$(\phi_\#S) \on B = \phi_\#(S \on \phi^{-1}(B))$ and
\[
\|\phi_\#S\|(B) \le \Lip(\phi)^n\,\|S\|(\phi^{-1}(B)).
\]
(See Sect.~4 in~\cite{Lan-LC}.)

\subsection*{Integral currents}
A current $S \in \cD_n(X)$ is {\em locally integer rectifiable\/}
if $\|S\|$ is locally finite and concentrated on the union of countably
many Lipschitz images of compact subsets of $\R^n$,
and for every Borel set $A \sub X$ with compact closure and every Lipschitz
map $\phi \colon \olA \to \R^n$, the current $\phi_\#(S \on A) \in \cD_n(\R^n)$
is of the form $\bb{w}$ for some {\em integer valued\/}
$w \in L^1(\R^n)$. Then $\|S\|$ is absolutely
continuous with respect to $n$-dimensional Hausdorff measure. 
Push-forwards and restrictions to Borel sets of locally integer
rectifiable currents are again locally integer rectifiable.

A current $S \in \cD_n(X)$ is called a {\em locally integral current\/} 
if $S$ is locally integer rectifiable and, for $n \ge 1$, $\|\d S\|$
is locally finite; then (by Theorem~8.7 in~\cite{Lan-LC}) $\d S$ is itself
locally integer rectifiable.
This gives a chain complex of abelian groups $\bI_{n,\loc}(X)$.
The subgroups $\bI_{n,\cs}(X)$ of {\em integral currents\/} consist
of the elements with compact support and, hence, finite total mass. 
For $X = \R^N$, there is a canonical chain isomorphism from $\bI_{*,\cs}(\R^N)$
to the chain complex of classical (Federer--Fleming)
integral currents~\cite{FedF} in $\R^N$.

For $n \ge 1$, we put $\bZ_{n,\cs}(X) := \{Z \in \bI_{n,\cs}(X): \d Z = 0\}$.
For $n = 0$, an element of $\bI_{0,\cs}(X)$ is an integral
linear combination of currents of the form $\bb{x}$, where
$\bb{x}(f_0) = f_0(x)$. We let $\bZ_{0,\cs}(X) \sub \bI_{0,\cs}(X)$ denote the
subgroup of linear combinations whose coefficients add up to zero.
The boundary of a current in $\bI_{1,\cs}(X)$ belongs to $\bZ_{0,\cs}(X)$.
Given $Z \in \bZ_{n,\cs}(X)$, for $n \ge 0$, we will call
$V \in \bI_{n+1,\cs}(X)$ a {\em filling} of $Z$ if $\d V = Z$.

\subsection*{Slicing}
Let $S \in \bI_{n,\loc}(X)$, $n \ge 1$, and let $\pi \colon X \to \R$
be a Lipschitz function. 
For $s \in \R$, the {\em slice} $T_s \in \cD_{n-1}(X)$ of $S$ with respect
to $\pi$ is the current 
\[
T_s := \d(S \on \{\pi \le s\}) - (\d S) \on \{\pi \le s\}
\]
with support in $\{\pi = s\} \cap \spt(S)$. Note that the restrictions
are defined since both $\|S\|$ and $\|\d S\|$ are locally finite.
For $a < b$, the coarea inequality 
\[
\int_a^b \M(T_s) \,ds \le \Lip(\pi)\,\|S\|(\{a < \pi < b\})
\]
holds, and if $\pi|_{\spt(S)}$ is proper, then $T_s \in \bI_{n-1,\cs}(X)$
for almost all $s \in \R$. (See Sect.~6 and Theorem~8.5 in~\cite{Lan-LC}.)

\subsection*{Convergence and compactness}
A sequence $(S_i)$ in $\cD_n(X)$ {\em converges weakly\/}
to a current $S \in \cD_n(X)$ if $S_i \to S$ pointwise as functionals
on $\cD^n(X)$. Then, for every open set $U \sub X$,
\[
\|S\|(U) \le \liminf_{i \to \infty} \|S_i\|(U),
\]
thus the mass is lower semicontinuous with respect to weak convergence.
Furthermore, weak convergence commutes with the boundary operator and
with push-forwards. For locally integral currents, the following
compactness theorem holds (see Theorem~8.10 in~\cite{Lan-LC}).

\begin{theorem} \label{thm:cptness}
Let $X$ be a proper metric space, and let $n \ge 1$.
If $(S_i)$ is a sequence in $\bI_{n,\loc}(X)$ such that
\[
  \sup_i(\|S_i\| + \|\d S_i\|)(K) < \infty
\]
for every compact set $K \sub X$, then some subsequence
$(S_{i_k})$ converges weakly to a current $S \in \bI_{n,\loc}(X)$.
\end{theorem}  

\subsection*{Isoperimetric inequality and Plateau problem}
Recall condition $(\CI_n)$ from the introduction.
Cone inequalities are instrumental for the proof of isoperimetric
inequalities of Euclidean type (compare Sect.~3.4 in~\cite{Gro-FRR}).
For $n \ge 1$, we say that $X$ satisfies $(\EII_n)$ if there is a constant
$\gam > 0$ such that every cycle $Z \in \bZ_{n,\cs}(X)$ has a filling
$V \in \bI_{n+1,\cs}(X)$ with mass
\[
  \M(V) \le \gam\,\M(Z)^{(n+1)/n}.
\]
To make the constants in $(\CI_n)$ and $(\EII_n)$ explicit, we will
write $(\CI_n)[c]$ and $(\EII_n)[\gam]$.
The following result was established in a more general form
in Theorem~1.2 in~\cite{Wen-EII}. 

\begin{theorem} \label{thm:eii}
For all $n \ge 1$ and $c > 0$ there is a constant $\gam > 0$ such that
for every proper metric space $X$, $(\CI_n)[c]$ implies $(\EII_n)[\gam]$.
\end{theorem}

(Here the quasi-convexity condition~{\rm (CI$_0$)} is actually not needed.)
By Theorem~\ref{thm:cptness} and a well-known application of~$(\EII_n)$
one gets the following existence result for minimizing integral currents
(see the proof of Theorem~2.4 in~\cite{KleL}).

\begin{theorem} \label{thm:plateau}
Let $Z \in \bZ_{n,\cs}(X)$, where $X$ is a proper metric space
satisfying $(\CI_0)$ if $n = 0$ and $(\EII_n)[\gam]$ if $n \ge 1$.  
Then there is a filling $V \in \bI_{n+1,\loc}(X)$ of $Z$ with mass
\[
  \M(V) = \inf\{\M(V'): V' \in \bI_{n+1,\loc}(X),\,\d V' = Z\} < \infty. 
\]
In fact, every such minimizing $V$ has compact support due to the following
lower density bound: if $x \in \spt(V)$, $r > 0$, and
$\B{x}{r} \cap \spt(Z) = \es$, then
\[
  \D_{x,r}(V) := \frac{\|V\|(\B{x}{r})}{r^{n+1}} \ge
  \del_0 := \begin{cases} 2 & \text{if $n = 0$,} \\
  (n+1)^{-(n+1)}\gam^{-n} & \text{if $n \ge 1$;} \end{cases}    
\]
thus $\spt(V)$ is within distance $(\M(V)/\del_0)^{1/(n+1)}$ from $\spt(Z)$.
\end{theorem}


\section{A variational argument} \label{sect:var}

We start with a slight modification and extension of an effective recent
approximation result, Proposition~4.2 in~\cite{HuaKS-1}.
The main conclusion is that for a cycle $Z$ and any $\eta > 0$ there is a
cycle $Z'$ with mass $\le \M(Z)$ such that $Z - Z'$ has a minimizing filling
with mass $\le \eta\,\M(Z)$ and $Z'$ satisfies a uniform {\em lower\/}
density bound at scales $\lesssim \eta$. This will be used in the proofs
of Theorem~\ref{thm:thin-cycles} and Theorem~\ref{thm:sii1}.
We show in addition that if $Z$ satisfies a uniform {\em upper\/} density bound
above some threshold radius, then the same holds for $Z'$;
see assertion~(5) below. This will be employed in
Theorem~\ref{thm:lii} (linear isoperimetric inequality).

\begin{proposition} \label{prop:approx}
Let $n \ge 1$ and $\gam > 0$. Suppose that $X$ is a proper metric space
satisfying $(\EII_n)[\gam]$ and, if $n \ge 2$, also $(\EII_{n-1})[\gam]$.
Then for every $Z \in \bZ_{n,\cs}(X)$ and $\eta > 0$ there exists a minimizing
$V \in \bI_{n+1,\cs}(X)$ such that the following holds for
\[
  Z' := Z - \d V, \quad \mu := \eta^{-1}\|V\| + \|Z'\|,
\]
and some constants $\alpha,\theta > 0$ depending only on $n$ and $\gam$:
\begin{enumerate}
\item $\mu(X) \le \M(Z)$, in particular $\M(Z') \le \M(Z)$
  and\/ $\M(V) \le \eta\,\M(Z)$;
\item
  $\D_{x,r}(Z') \ge \theta$ for all $x \in \spt(Z')$ and $r \in (0,\alpha\eta]$;
\item
  if $B \sub X$ is a closed set and $T := \d(V \on B) - (\d V) \on B$ is
  in $\bI_{n,c}(X)$, then $\mu(B) \le \|Z\|(B) + \M(T)$;
\item
  if\/ $\M(Z) < m := \theta(\alpha\eta)^n$, then $Z' = 0$, and
  if\/ $\M(Z) \ge m$, then $\spt(Z')$ is within distance
  at most $\eta (\alpha + \ln(\M(Z)/m))$ from $\spt(Z)$;
\item
  if there exist $C > 0$, $a \ge 0$, and $p \in X$ such that $\D_{p,r}(Z) \le C$
  for all $r > a$,
  then $\mu(\B{p}{r}) \le 2^{n+1}Cr^n$ for all $r > \max\{a,2^{n+1}\eta\}$.
\end{enumerate}
\end{proposition}

\begin{proof}
Given $Z \in \bZ_{n,\cs}(X)$ and $\eta > 0$, consider the functional  
\[
  F \colon \bI_{n+1,\loc}(X) \to [0,\infty], \quad
  F(V') = \eta^{-1}\M(V') + \M(Z - \d V').
\]
Notice that $F$ is lower semicontinuous with respect to weak convergence,
like~$\M$. Moreover, $\M(V') \le \eta F(V')$
and $\M(\d V') \le F(V') + \M(Z)$ for all~$V'$, and $F(0) = \M(Z) < \infty$. 
We can thus pick a minimizing sequence for $F$ and use
Theorem~\ref{thm:cptness} to find a $V \in \bI_{n+1,\loc}(X)$ that minimizes
$F$. Now if $Z' := Z - \d V$ and $\mu := \eta^{-1}\|V\| + \|Z'\|$, then
\[
  \mu(X) = \eta^{-1}\M(V) + \M(Z') = F(V) \le F(0) = \M(Z),
\]  
so (1) holds. However, we still have to show that in fact
$V \in \bI_{n+1,\cs}(X)$.

We proceed with~(2). Let $x \in \spt(Z')$.
Put $f(s) := \|Z'\|(\B{x}{s}) > 0$ for all $s > 0$.
For almost every $s$, the slice $R_s := \d(Z' \on \B{x}{s})$ is in
$\bZ_{n-1,\cs}(X)$ and satisfies $\M(R_s) \le f'(s)$.
Suppose first that $n \ge 2$.
Then by the isoperimetric inequality there exists a filling
$T_s \in \bI_{n,\cs}(X)$ of $R_s$ such that 
\[
  \M(T_s) \le \gam\,\M(R_s)^{n/(n-1)} \le \gam f'(s)^{n/(n-1)}.
\]
Furthermore, the cycle $Z' \on \B{x}{s} - T_s$ has a filling
$W_s \in \bI_{n+1,\cs}(X)$ with
\[
  \M(W_s) \le \gam \bigl( f(s) + \M(T_s) \bigr)^{(n+1)/n}.  
\]
Since $Z - \d(V + W_s) = Z' - \d W_s = Z' \on (X \sm \B{x}{s}) + T_s$,
we have
\begin{align*}
  F(V + W_s)
  &= \eta^{-1}\M(V + W_s) + \M(Z' \on (X \sm \B{x}{s}) + T_s) \\
  &\le \eta^{-1}(\M(V) + \M(W_s)) + \|Z'\|(X \sm \B{x}{s}) + \M(T_s).
\end{align*}
It follows that $0 \le F(V + W_s) - F(V) \le \eta^{-1}\M(W_s) - f(s) + \M(T_s)$
and  
\[
  \eta \bigl( f(s) - \M(T_s) \bigr) \le \M(W_s)
  \le \gam \bigl( f(s) + \M(T_s) \bigr)^{(n+1)/n}.
\]
Hence, if $\M(T_s) \le f(s)/2$, then
$\eta f(s)/2 \le \gam (3f(s)/2)^{(n+1)/n}$ and thus
\[
  f(s) \ge \theta' \eta^n
\]
for $\theta' := 2/(3^{n+1}\gam^n)$.
Now if $f(r) < \theta' \eta^n$ for some $r$,
then $f(s) < \theta' \eta^n$ for all $s \in (0,r)$, thus
$f(s)/2 < \M(T_s) \le \gam f'(s)^{n/(n-1)}$ and
\[
f'(s)f(s)^{(1-n)/n} \ge (2\gam)^{(1-n)/n}
\]
for almost every such $s$, and integration from $0$ to $r$ yields
\[
f(r) \ge \theta r^n
\]
where $\theta := n^{-n}(2\gam)^{1-n}$. This shows that
$f(r) \ge \min\{\theta' \eta^n, \theta r^n\}$ for all $r > 0$.
Hence~(2) holds with $\alpha := (\theta'/\theta)^{1/n}$ in case $n \ge 2$.
Suppose now that $n = 1$. Then every non-zero slice $R_s \in \bZ_{0,\cs}(X)$
has mass at least $2$. If $R_s = 0$ for some $s$, then $Z' \on \B{x}{s}$
is a cycle, and repeating the above argument with $T_s = 0$ we get that
$\eta f(s) \le \M(W_s) \le \gam f(s)^2$, thus $f(s) \ge \eta/\gam$.
Hence, if $f(r) < \eta/\gam$ for some $r$, then $f'(s) \ge \M(R_s) \ge 2$
for almost every $s \in (0,r)$ and so $f(r) \ge 2r$.
We conclude that in case $n = 1$, (2) holds with $\theta := 2$ and
$\alpha := 1/(2\gam)$.

Since $\M(Z') < \infty$, it now follows from~(2) that $Z'$ has compact 
support. Hence $\spt(\d V)$ is compact, and since $V$ has finite mass
and is evidently minimizing, by Theorem~\ref{thm:plateau} $\spt(V)$ is
compact as well.

We prove~(3). Let $W := V \on B$.
If $T = \d W - (Z - Z') \on B \in \bI_{n,\cs}(X)$,
then also $(Z - Z') \on B \in \bI_{n,\cs}(X)$ and
$W \in \bI_{n+1,\cs}(X)$. Since  
$Z - \d(V - W) = Z' + \d W = Z' \on (X \sm B) + Z \on B + T$,
we have
\begin{align*}
F(V - W)
  &= \eta^{-1}\M(V - W)
    + \M \bigl( Z' \on (X \sm B) + Z \on B + T \bigr) \\
  &\le \mu(X \sm B) + \|Z\|(B) + \M(T).
\end{align*} 
Since $\mu(X) = F(V) \le F(V - W)$, it follows that
$\mu(B) = \mu(X) - \mu(X \sm B) \le \|Z\|(B) + \M(T)$ as claimed.  

The first assertion of~(4) is clear from~(1) and~(2).
Suppose now that $\M(Z) \ge m = \theta(\alpha\eta)^n$ and
$x \in \spt(Z')$ is a point at distance $D > \alpha\eta$ from $\spt(Z)$.
Set $g(s) := \mu(\B{x}{s})$ for all $s \in (0,D)$. For almost every such $s$,
the slice $T_s := \d(V \on \B{x}{s}) + Z' \on \B{x}{s}$ is in $\bI_{n,\cs}(X)$
and satisfies $\M(T_s) \le \frac{d}{ds} \|V\|(\B{x}{s})$
as well as $\M(\d T_s) \le \frac{d}{ds} \|Z'\|(\B{x}{s})$.
Hence, by~(3),
\[
  g(s) \le \M(T_s) \le \M(T_s) + \eta\,\M(\d T_s) \le \eta\,g'(s). 
\]
Integrating the inequality $1 \le \eta\,g'(s)/g(s)$ from $\alpha\eta$
to $t < D$ we get that
\[
  t \le \eta \bigl( \alpha + \ln(g(t)) - \ln(g(\alpha\eta)) \bigr).
\]
By~(1) and~(2), $g(t) \le \mu(X) \le \M(Z)$ and $g(\alpha\eta)
\ge \|Z'\|(\B{x}{\alpha\eta}) \ge m$. 
As this holds for all $t < D$, (4) follows. 

It remains to prove~(5). We will write $B_r$ for $\B{p}{r}$.
First we choose a sufficiently large $r_0 > 0$ so that
\[
\|V\|(B_{r_0}) \le 2^{n+1}\eta\,Cr_0^{\,n},
\]
then we put $r_i := 2^{-i}r_0$ for every integer $i \ge 1$.
There exists an $s \in (r_1,r_0)$ such that the slice
$T_s := \d(V \on B_s) - (\d V) \on B_s\in \bI_{n,\cs}(X)$ satisfies
\[
  \mu(B_s) \le \|Z\|(B_s) + \M(T_s) 
\]
by~(3), as well as
$\M(T_s) \le \|V\|(B_{r_0})/(r_0 - r_1) \le 2^{n+1}\eta\,Cr_0^{\,n}/r_1$.
Now if $r_1 > \max\{a,2^{n+1}\eta\}$,
then $\|Z\|(B_s) \le Cs^n$ and $\M(T_s) \le Cr_0^{\,n}$, hence
\[
  \mu(B_{r_1}) \le \mu(B_s) \le 2Cr_0^{\,n} = 2^{n+1}Cr_1^{\,n}.
\]
This also yields the above inequality for the next smaller scale, 
\[
\|V\|(B_{r_1}) \le 2^{n+1}\eta\,Cr_1^{\,n}.
\]
Finally, given any $r > \max\{a,2^{n+1}\eta\}$, we can choose
$r_0$ initially such that $r = r_k = 2^{-k}r_0$ for some $k \ge 1$.  
When $k \ge 2$, we repeat the above slicing argument 
successively for $i = 2,\dots,k$,
with $(r_i,r_{i-1})$ in place of $(r_1,r_0)$.
This eventually shows that
\[
  \mu(B_r) \le 2^{n+1}Cr^n
\]
for any $r > \max\{a,2^{n+1}\eta\}$. 
\end{proof}

As a first application of parts~(1)--(3) of Proposition~\ref{prop:approx}
we give a short proof of a variant of Proposition~5.8 in~\cite{Wen-FC}
regarding fillings of thin cycles. This result will play
a key role in the next section (see Theorem~\ref{thm:w-fillvol}).
The assumptions on $X$ are as above. The conclusion is that
for a cycle $Z \in \bZ_{n,\cs}(X)$, the {\em filling volume}
\[
  \FillVol_X(Z) := \inf\{\M(V): V \in \bI_{n+1,\cs}(X),\,\d V = Z\}
\]
can be forced to be arbitrarily small by imposing a sufficiently small
bound on the supremal mass
\[
  m_\rho(Z) := \sup_{x \in X} \|Z\|(\B{x}{\rho}) 
\]
in all closed balls of some fixed radius $\rho > 0$. The proof gives
an explicit constant involving a mass bound for $Z$. 

\begin{theorem} \label{thm:thin-cycles}
For all $n \ge 1$ and $\gam,M,\rho,\nu > 0$ there exists a constant
$\del = \del(n,\gam,M,\rho,\nu) > 0$ such that if $X$ is a proper metric space
satisfying $(\EII_n)[\gam]$ and, in case $n \ge 2$, also $(\EII_{n-1})[\gam]$,
then every $Z \in \bZ_{n,\cs}(X)$ with $\M(Z) \le M$ and $m_\rho(Z) \le \del$
has $\FillVol_X(Z) < \nu$.
\end{theorem}

\begin{proof}
Let $\alpha$ and $\theta$ be the constants
from Proposition~\ref{prop:approx}, depending on $n$ and $\gam$.
Given $M,\rho,\nu > 0$, fix $\eta > 0$ such that both $\eta M$ and
$\gam(4\eta M/\rho)^{(n+1)/n}$ are less than $\nu/2$, then put
\[
  r := \min \Bigl\{ \alpha \eta,\frac{\rho}{4} \Bigr\}, \quad
  \del := \frac12 \theta r^n.
\]
Suppose now that $Z \in \bZ_{n,\cs}(X)$ satisfies $\M(Z) \le M$ and
$m_\rho(Z) \le \del$. By Proposition~\ref{prop:approx} there exists
$V \in \bI_{n+1,\cs}(X)$ such that, for $Z' := Z - \d V$,
\begin{enumerate}
\item $\M(V) \le \eta\,\M(Z) \le \eta M < \nu/2$;
\item $\|Z'\|(\B{x}{r}) \ge \theta r^n = 2\del$ for all
  $x \in \spt(Z')$;  
\item $\|Z'\|(B) \le \|Z\|(B) + \M(T)$ whenever $B \sub X$ is a closed set
  and $T := \d(V \on B) - (\d V) \on B$ is in $\bI_{n,\cs}(X)$.
\end{enumerate}
If $Z' = 0$, then $\d V = Z$, and (1)~yields the result.
Now let $Z' \ne 0$. It remains to show that $\FillVol_X(Z') < \nu/2$.
Pick a maximal set $N \sub \spt(Z')$ of distinct points at mutual
distance $> 2r$, and put $B_s := \bigcup_{x \in N}\B{x}{s}$
for all $s > 0$.
By~(2), since the balls $\B{x}{r}$ with $x \in N$ are pairwise disjoint,
we have $2\del\,|N| \le \|Z'\|(B_r) \le \M(Z')$, and so
\[
  \|Z\|(B_\rho) \le |N|\,m_\rho(Z) \le \del\,|N| \le \frac12\,\M(Z').
\]
Furthermore, since $N$ is maximal, $\spt(Z') \sub B_{2r} \sub B_{\rho/2}$.
Hence, for almost every $s \in (\frac{\rho}{2},\rho)$, the slice 
$T_s := \d(V \on B_s) - (\d V) \on B_s \in \bI_{n,\cs}(X)$ satisfies
\[
  \M(Z') = \|Z'\|(B_s) \le \|Z\|(B_s) + \M(T_s) \le \frac12\,\M(Z') + \M(T_s)
\]
by~(3). We conclude that $\M(T_s) \ge \M(Z')/2$ and
\[
  \M(V) \ge \|V\|(B_\rho) \ge \int_{\rho/2}^\rho \M(T_s) \,ds
  \ge \frac{\rho}{4}\,\M(Z').
\]
It now follows from~(1) that $\M(Z') \le 4\eta M/\rho$,
and by the isoperimetric inequality and the choice of $\eta$
we get that $\FillVol_X(Z') < \nu/2$.
\end{proof}  


\section{Convergence of cycles} \label{sect:conv}

A central result in geometric measure theory says that a weakly convergent
sequence $S_i \to S$ of integral $n$-currents with supports in a fixed compact
set and with $\sup_i(\M(S_i) + \M(\d S_i)) < \infty$ converges in the 
{\em flat metric topology}. This means that there exist integral $n$-currents
$T_i$ and integral $(n+1)$-currents $V_i$ such that $S_i - S = T_i + \d V_i$
and $\M(T_i) + \M(V_i) \to 0$. In~$\R^N$, this property can be deduced
from the deformation theorem (see Theorem~5.5 and Theorem~7.1 in~\cite{FedF}).
The result was generalized in~\cite{Wen-FC} to Ambrosio--Kirchheim currents
in complete metric spaces satisfying condition~$(\CI_n)$ locally.
It essentially suffices to show that $\FillVol_X(Z_i) \to 0$
for any bounded sequence of cycles $Z_i$ converging weakly to $0$.
In this section we prove a somewhat amplified version of this fact
for proper metric spaces satisfying~$(\CI_n)$ globally, so as to
facilitate the proof of the sub-Euclidean isoperimetric inequality
in the next section.
The argument relies on Theorem~\ref{thm:thin-cycles} and proceeds along
the same lines as~\cite{Wen-FC}, but is simplified by the use of a uniform
notion of weak convergence. 

Let $S \in \cD_n(X)$, $n \ge 0$, and suppose that $\spt(S)$ is compact.
Note that every such $S$ extends canonically to tuples of Lipschitz functions
whose first entry is no longer required to have compact support. 
We define
\[
  \W(S) := \sup\{ S(f_0,\ldots,f_n) :
  \text{$f_0,\ldots,f_n$ are $1$-Lipschitz, $|f_0| \le 1$} \}.
\]
Evidently $\W(S) \le \M(S)$, and if $n \ge 1$, then $\W(\d S) \le \W(S)$.

The following auxiliary result is an adaptation of Proposition~6.6
in~\cite{Lan-LC} to sequences of cycles in possibly distinct proper
metric spaces. 

\begin{lemma} \label{lem:seq-slice}
Suppose that $n \ge 1$, $Z_i \in \bZ_{n,c}(X_i)$, $\sup_i\M(Z_i) < \infty$,
$\W(Z_i) \to 0$, and $\pi_i \colon X_i \to \R$ is $1$-Lipschitz.
Then for almost every $s \in \R$ there is a sequence $(i_k)$ such that
$Z_{i_k,s} := Z_{i_k} \on \{\pi_{i_k} \le s\} \in \bI_{n,\cs}(X_{i_k})$ and
\[
\sup_k \M(\d Z_{i_k,s}) < \infty, \quad
\W(\d Z_{i_k,s}) \le \W(Z_{i_k,s}) \to 0 \quad (k \to \infty).
\]
\end{lemma}

\begin{proof}  
Consider the Borel measures $\mu_i := \pi_{i\#}\|Z_i\|$. Since 
$\sup_i\mu_i(\R) < \infty$, some subsequence $(\mu_{i_k})$ converges weakly
to a finite Borel measure $\mu$ on $\R$. Furthermore, 
for the slices $\d Z_{i_k,s}$,
\[
  \int_{\R} \liminf_{k \to \infty} \M(\d Z_{i_k,s})\,ds
  \le \liminf_{k \to \infty} \int_{\R} \M(\d Z_{i_k,s})\,ds
  \le \sup_i\M(Z_i) < \infty.
\]
We now take $s$ so that $\mu(\{s\}) = 0$ and
$\liminf_{k \to \infty}\M(\d Z_{i_k,s}) < \infty$, then we adjust the sequence
$(i_k)$, if necessary, to arrange that $\sup_k \M(\d Z_{i_k,s}) < \infty$.  
Note that $Z_{i_k,s} \in \bI_{n,\cs}(X_{i_k})$ for all~$k$.
Let $\eps > 0$, and choose $\del > 0$ such that $\mu([s,s+\del]) < \eps$.
Let $\gam_{s,\del} \colon \R \to \R$ denote the
piecewise affine $\del^{-1}$-Lipschitz function that is $1$ on $(-\infty,s]$ and
$0$ on $[s+\del,\infty)$, and put $u_k := \gam_{s,\del} \circ \pi_{i_k}$
and $v_k := \chi_{(-\infty,s]} \circ \pi_{i_k}$.
If $k$ is sufficiently large, then 
\begin{gather*}
  \W(Z_{i_k} \on (u_k - v_k)) \le \mu_{i_k}([s,s+\del]) < \eps, \\
  \W(Z_{i_k} \on u_k) \le (1+\del^{-1})\W(Z_{i_k}) < \eps
\end{gather*}
(note that if $f_0 \colon X_{i_k} \to \R$ is a $1$-Lipschitz function
with $|f_0| \le 1$, then $u_kf_0$ is $(1 + \del^{-1})$-Lipschitz),
thus $\W(Z_{i_k,s}) = \W(Z_{i_k} \on v_k) < 2\eps$.
This gives the result.
\end{proof}  

Next, recall that a $0$-cycle $Z \in \bZ_{0,\cs}(X)$ is of the form
$Z(f) = \sum_i a_if(x_i)$ for finitely many points $x_i \in X$ and weights
$a_i \in \Z$ with $\sum_i a_i = 0$. Note that $Z(f + c) = Z(f)$ for all
$c \in \R$. We define
\[
  \cW(Z) := \sup\{Z(f): \text{$f$ is $1$-Lipschitz}\};
\]
in general $\cW(Z) \ge \W(Z)$. We have the following optimal result
for $n = 0$ (compare Theorem~1.2 in~\cite{Wen-FC}).
  
\begin{proposition} \label{prop:w-fillvol-0}
Let $X$ be a proper metric space satisfying~$(\CI_0)[c]$, that is,
any two points $x,y \in X$ can be connected by a curve of length
$\le c\,d(x,y)$. Then every cycle $Z \in \bZ_{0,\cs}(X)$ has
$\FillVol_X(Z) \le c\,\cW(Z)$, and if\/ $\W(Z) < 2$, then $\cW(Z) = \W(Z)$.
\end{proposition}

\begin{proof}
We can write $Z \ne 0$ in the form
\[
  Z(f) = \sum_{i=1}^k (f(x_i) - f(y_i)) = \int_X f \,d\mu - \int_X f \,d\nu
\]
for some (not necessarily distinct) points $x_1,\ldots,x_k$ and
$y_1,\ldots,y_k$ in $X$ and the corresponding measures
$\mu := \sum_i\del_{x_i}$ and $\nu := \sum_i\del_{y_i}$.
Hence, by the Kantorovich--Rubinstein theorem (see~\cite{Edw}),
$\cW(Z)$ equals the Wasserstein distance $W_1(\mu,\nu)$.
It is well-known that for such measures the latter agrees with the
minimum of $\sum_i d(x_i,y_{\pi(i)})$ over all permutations $\pi$ of
$\{1,\ldots,k\}$.
Thus, some such sum is equal to $\cW(Z)$. We will give an alternative direct
proof of this identity in Lemma~\ref{lem:w} below.  
It now follows from condition $(\CI_0)[c]$ that $Z$ has a filling with mass
less than or equal to $c\,\cW(Z)$.

For the second assertion of the proposition,
consider the metric $\del := \min\{d,2\}$ on $X$
and let $\cW_{\del}$ denote the corresponding functional.
Note that $\cW_\del(Z) = \W(Z)$. If $Z$ is as above, then for some
$1$-Lipschitz function $f \colon (X,\del) \to \R$ and some permutation $\pi$,
we have
\[
  \cW_\del(Z) = Z(f)
  = \sum_i (f(x_i) - f(y_{\pi(i)}))
  = \sum_i \del(x_i,y_{\pi(i)}),
\]
in particular $f(x_i) - f(y_{\pi(i)}) = \del(x_i,y_{\pi(i)}) \ge 0$ for all $i$.
We can assume that the set $\bigcup_i[f(y_{\pi(i)}),f(x_i)]$ is connected;
otherwise $f$ can easily be modified so that this holds.
Hence, if $\cW_\del(Z) = \W(Z) < 2$,
then we can further arrange that $|f| < 1$. It follows that there is no
$1$-Lipschitz function $g \colon X \to \R$ with $Z(g) > \W(Z) = Z(f)$, for
otherwise a suitable convex combination $h = (1-\eps)f + \eps g$ would
satisfy $Z(h) > \W(Z)$ and $|h| \le 1$.
\end{proof}  

We now provide the alternative argument mentioned above. It is convenient
to consider pairwise distinct points but to allow distances to be zero.

\begin{lemma} \label{lem:w}
Let $(V,d)$ be a finite pseudo-metric space with a partition $V = V_+ \cup V_-$,
where $|V_+| = |V_-|$. If $f \colon V \to \R$ is a
$1$-Lipschitz function that maximizes the quantity
$\sum_{x \in V_+}f(x) - \sum_{y \in V_-}f(y)$,
then there exists a bijection $\pi \colon V_+ \to V_-$ such that
$f(x) - f(\pi(x)) = d(x,\pi(x))$ for all $x \in V_+$.
\end{lemma}

\begin{proof}
Given $f$, define a relation $\preceq$ on $V$ such that $x \preceq y$
if and only if $f(x) - f(y) = d(x,y)$. Note that if $x \preceq y \preceq z$,
then
\[
  f(x) - f(z) = d(x,y) + d(y,z) \ge d(x,z),
\]
and since $f$ is $1$-Lipschitz, $x \preceq z$.
Thus the relation is transitive. For a set $A \sub V_+$, let $\Gam(A)$
denote the set of all $y \in V_-$ for which there is an $x \in A$ with
$x \preceq y$. Suppose first that $A$ is maximal in $V_+$ in the sense
that there is no pair $(x,y) \in A \times (V_+ \sm A)$ with $x \preceq y$.
Note that $f(x) - f(y) < d(x,y)$ whenever
$x \in A$ and $y \in C := (V_+ \sm A) \cup (V_- \sm \Gam(A))$.
By transitivity, the same strict inequality holds whenever $x \in \Gam(A)$
and $y \in C$. Hence, for some $\eps > 0$, the function $f_\eps$ obtained
from $f$ by increasing the values on $A \cup \Gam(A)$ by $\eps$ is still
$1$-Lipschitz. It follows from the maximality property of $f$ that
\[
  \eps\,|A| = \sum_{x \in V_+} (f_\eps(x) - f(x))
  \le \sum_{y \in V_-} (f_\eps(y) - f(y)) = \eps\,|\Gam(A)|,
\]
thus $|A| \le |\Gam(A)|$. Let now $A \sub V_+$ be arbitrary.
Again by transitivity, the set $A'$ of all points in $V_+$ with a precursor
in $A$ is maximal, and $\Gam(A') = \Gam(A)$, so that
$|A| \le |A'| \le |\Gam(A')| = |\Gam(A)|$. This shows that the bipartite
graph with edge set $\{(x,y) \in V_+ \times V_-: x \preceq y\}$ satisfies
the assumption of Hall's marriage theorem.
Hence, there is a matching (bijection) $\pi$ as claimed.
\end{proof}  
 
In general, for $n \ge 0$, an analog of
Proposition~\ref{prop:w-fillvol-0} holds as follows.

\begin{theorem} \label{thm:w-fillvol}
If $X_i$ satisfies $(\CI_n)[c]$ for $i \in \N$, and if the cycles
$Z_i \in \bZ_{n,c}(X_i)$ satisfy $\sup_i\M(Z_i) < \infty$
and\/ $\W(Z_i) \to 0$, then $\FillVol_{X_i}(Z_i) \to 0$.
\end{theorem}  

\begin{proof}
The proof is by induction on $n$. For $n = 0$, the result holds by
Proposition~\ref{prop:w-fillvol-0}.
Assume now that $n \ge 1$ and the assertion holds in dimension $n-1$.
It suffices to show that for every sequence $(Z_i)_{i \in \N}$ as in the
statement and for every $\nu > 0$ there is an index $j$ with
$\FillVol_{X_j}(Z_j) < \nu$.
By Theorem~\ref{thm:eii} there is a constant $\gam = \gam(n,c)$ such that
every $X_i$ satisfies $(\EII_n)[\gam]$ and, if $n \ge 2$, also
$(\EII_{n-1})[\gam]$. Put $M := \sup_i\M(Z_i)$ and choose $\rho > 0$ such that
\[
  18c\rho M < \nu.
\]
Let $\del := \del(n,\gam,M,\rho,\nu/2)$
be the constant from Theorem~\ref{thm:thin-cycles}.
If there is an index $j$ with $m_\rho(Z_j) \le \del$, then
$\FillVol_{X_j}(Z_j) < \nu/2$ and we are done.

Suppose now that $m_\rho(Z_i) > \del$ for all~$i$.
Choose points $x_i \in X_i$ such that
\[
  \|Z_i\|(\B{x_i}{\rho}) \ge \del.
\]
By Lemma~\ref{lem:seq-slice} there is an $s \in (\rho,2\rho)$ and an infinite
set $I_1 \sub \N$ such that $Z_{i,s} := Z_i \on \B{x_i}{s} \in \bI_{n,\cs}(X_i)$
for all $i \in I_1$ and
\[
  \sup_{i \in I_1} \M(\d Z_{i,s}) < \infty,
  \quad \W(\d Z_{i,s}) \le \W(Z_{i,s}) \to 0 \quad (I_1 \ni i \to \infty).
\]
By the induction assumption there exist fillings
$T_i \in \bI_{n,\cs}(X_i)$ of $\d Z_{i,s}$ such that $\M(T_i) \to 0$.
By reducing the index set further, if necessary, we arrange that
$\M(T_i) \le \del/2$ and $\spt(T_i) \sub \B{x_i}{3\rho}$
for all $i \in I_1$ (Theorem~\ref{thm:plateau}).
Then $S_i^1 := Z_{i,s} - T_i$ is a cycle with support in $\B{x_i}{3\rho}$, and
\[
  \FillVol_{X_i}(S_i^1) \le 3c\rho\,\M(S_i^1)
\]
by the cone inequality. Let $Z_i^1 := Z_i - Z_{i,s} + T_i$ and consider the
splitting
\[
  Z_i = S_i^1 + Z_i^1.
\]
Notice that $\M(S_i^1) \le \M(Z_{i,s}) + \del/2$ and
$\M(Z_i^1) \le \M(Z_i) - \M(Z_{i,s}) + \del/2$,
moreover $\M(Z_{i,s}) \ge \|Z_i\|(\B{x_i}{\rho}) \ge \del$.
Hence, for all $i \in I_1$, we have
\begin{align*}
  \M(S_i^1) &\le \M(Z_i) - \M(Z_i^1) + \del, \\
  \M(Z_i^1) &\le \M(Z_i) - \frac{\del}{2} \le M.
\end{align*}
Note further that 
$\W(Z_i^1) \le \W(Z_i) + \W(Z_{i,s}) + \M(T_i) \to 0$ as $i \to \infty$.

If $m_\rho(Z_i^1) > \del$ for all $i \in I_1$, then we repeat the above
argument (with the same constants $M,\rho,\del$) and produce similar
splittings $Z_i^1 = S_i^2 + Z_i^2$ for all~$i$ in an infinite set
$I_2 \sub I_1$. If $m_\rho(Z_i^2) > \del$ for all $i \in I_2$, we iterate
this step, and continue in this manner. This eventually
yields an infinite set $I_k \sub \N$, for some $k \ge 1$, and a decomposition
\[
  Z_i = S_i^1 + \ldots + S_i^k + Z_i^k
\]
for every $i \in I_k$,
such that $m_\rho(Z_j^k) \le \del$ for some $j \in I_k$. 
In fact, $k \le 2M/\del$, because
$\M(Z_i^k) \le \M(Z_i) - k \del/2$ for all $i \in I_k$.
It follows that
\begin{gather*}
  \M(S_i^1) + \ldots + \M(S_i^k) \le \M(Z_i) - \M(Z_i^k) + k\del \le 3M, \\
  \FillVol_{X_i}(S_i^1) + \ldots + \FillVol_{X_i}(S_i^k) \le 9c\rho M
  < \frac{\nu}{2},
\end{gather*}
and $\FillVol_{X_j}(Z_j^k) < \nu/2$ by
Theorem~\ref{thm:thin-cycles}. Hence, $\FillVol_{X_j}(Z_j) < \nu$.
\end{proof}

The desired result for sequences of cycles with supports in a fixed
compact set now follows easily.

\begin{theorem} \label{thm:weak-flat}
Let $X$ be a proper metric space satisfying condition $(\CI_n)$ for some
$n \ge 0$. If an $\M$-bounded sequence of cycles $Z_i \in \bZ_{n,\cs}(X)$
with supports in a fixed compact set $K \sub X$ converges weakly to $0$,
then $\W(Z_i) \to 0$ and $\FillVol_X(Z_i) \to 0$.  
\end{theorem}

\begin{proof}
Suppose that $\M(Z_i) \le M$ for all $i$.
Let $\cF$ denote the collection of all $1$-Lipschitz functions
$f \colon X \to \R$ with $|f| \le \diam(K)/2$. Let $\eps > 0$.
There is a finite subcollection $\cG \sub \cF$ such that for all
$f_0,\ldots,f_n \in \cF$ there exist $g_0,\ldots,g_n \in \cG$ with
$\sup_{x \in K}|f_k(x) - g_k(x)| \le \eps/M$ for $k = 0,\ldots,n$; then
\[
  |Z_i(f_0,\ldots,f_n) - Z_i(g_0,\ldots,g_n)| \le (n+1)\eps
\]
for all $i$ by Lemma~5.2 in~\cite{Lan-LC}.
As $Z_i \to 0$ weakly, if $i$ is sufficiently large, then
$Z_i(g_0,\ldots,g_n) \le \eps$ for all tuples $(g_0,\ldots,g_n) \in \cG^{n+1}$,
thus $Z_i(f_0,\ldots,f_n) \le (n+2)\eps$ whenever $f_0,\ldots,f_n \in \cF$.
Hence $\W(Z_i) \to 0$, and $\FillVol_X(Z_i) \to 0$ by
Theorem~\ref{thm:w-fillvol}.
\end{proof}  


\section{Isoperimetric inequalities} \label{sect:ii}

This section is devoted to isoperimetric inequalities and shows
in particular the implications
$(\AR_n) \Rightarrow (\LII_n) \Leftrightarrow (\SII_n)$
in Theorem~\ref{thm:main}.
In fact we prove some stronger uniform statements. To this end we first
extend the notion of asymptotic rank to sequences of metric spaces
$X_i = (X_i,d_i)$.
A compact metric space $\Om$ is called an {\em asymptotic subset\/}
of the sequence $(X_i)_{i \in \N}$ if there exist positive numbers
$r_i \to \infty$ and subsets $Y_i \sub X_i$ such that the rescaled sets
$(Y_i,r_i^{-1}d_i)$ converge to~$\Om$ in the Gromov--Hausdorff topology.
The {\em asymptotic rank\/} of the sequence $(X_i)$ is the supremum of all
$k \ge 0$ such that there exists an asymptotic subset bi-Lipschitz
homeomorphic to a compact subset of $\R^k$ with positive Lebesgue measure.
The asymptotic rank of a single metric space $X$,
as defined in the introduction, equals the asymptotic rank
of the constant sequence $X_i = X$ (see Definition~1.1 and Proposition~3.1
in~\cite{Wen-AR}).

From now on, throughout this section, we assume that $\cX$ is a class of 
proper metric spaces such that for some $n \ge 1$ and $c > 0$,
all members of $\cX$ satisfy $(\CI_n)[c]$, and every sequence
$(X_i)_{i \in \N}$ in $\cX$ has asymptotic rank $\le n$.
The sub-Euclidean isoperimetric inequality for $n$-cycles in spaces
of asymptotic rank at most $n$ was established in greater generality
in~\cite{Wen-AR} (Theorem~1.2), and a slightly restricted version was used
as a key tool in~\cite{KleL} (Theorem~4.4). First we give a proof
of a uniform version of the latter statement. 

\begin{theorem} \label{thm:sii1}  
For all $C,\eps > 0$ there is a constant $\rho_0 = \rho_0(\cX,n,c,C,\eps) > 0$
such that if $X$ belongs to $\cX$, and $Z \in \bZ_{n,\cs}(X)$ satisfies
$\M(Z) \le C r^n$ and $\spt(Z) \sub \B{p}{r}$ for some $p \in X$
and $r \ge \rho_0$, then $\FillVol_X(Z) < \eps r^{n+1}$.
\end{theorem}

\begin{proof}
Suppose to the contrary that there exist $C,\eps > 0$, 
a sequence of positive radii $(r_i)_{i \in \N}$ tending to infinity,
and cycles $Z_i \in \bZ_{n,\cs}(X_i)$, where $X_i = (X_i,d_i)$ belongs to $\cX$,
each with mass $\M(Z_i) \le C r_i^{\,n}$ and support in some ball $\B{p_i}{r_i}$,
such that
\[
  \FillVol_{X_i}(Z_i) \ge \eps r_i^{\,n+1}.
\]
By Theorem~\ref{thm:eii} there is a constant $\gam = \gam(n,c)$ such that
every $X_i$ satisfies $(\EII_n)[\gam]$ and, if $n \ge 2$,
also $(\EII_{n-1})[\gam]$.
Put $\eta_i := \eps r_i/(2C)$ and apply Proposition~\ref{prop:approx}
to $Z_i$ to get $V_i \in \bI_{n+1,\cs}(X_i)$ and $Z_i' := Z_i - \d V_i$ such that
\begin{enumerate}
\item
  $\M(Z_i') \le C r_i^{\,n}$ and
  $\M(V_i) \le \eta_i C r_i^{\,n} = \eps r_i^{\,n+1}/2$; 
\item
  $\D_{x,r}(Z_i') \ge \theta$ whenever $x \in \spt(Z_i')$ and
  $0 < r \le \alpha\eta_i = \alpha\eps r_i/(2C)$, where $\alpha,\theta$
  depend only on $n,\gam$;
\item
  $\spt(Z_i') \sub \B{p_i}{\lam r_i}$ for some constant $\lam > 1$ depending
  only on $n,\gam,C,\eps$ (this uses part~(4) of
  Proposition~\ref{prop:approx}).
\end{enumerate}
By (1), (3) and the coning inequality $(\CI_n)[c]$ there exists a filling
$V_i' \in \bI_{n+1,\cs}(X_i)$ of $Z_i'$ with mass
\[
  \M(V_i') \le c\lam r_i\,\M(Z_i') \le c\lam C r_i^{\,n+1}.
\]
By Theorem~\ref{thm:plateau} we can assume that $V_i'$ is minimizing and has
support in $\B{p_i}{\lam'r_i}$ for some $\lam' > \lam$ independent of $i$.
Let $Y_i$ denote the set $\spt(V_i')$ equipped with the metric induced
by $r_i^{-1}d_i$. Note that $\M(Z_i') \le C$ and $\M(V_i') \le c\lam C$
with respect to this metric, and $Y_i$ has diameter at most $2\lam'$.
It follows from~(2) and the lower density bound for $V_i'$
that the family of all $Y_i$ is uniformly precompact.
By Gromov's compactness theorem~\cite{Gro-PG}, after passage to subsequences
and relabelling, there exist a compact metric space $Y$ and
isometric embeddings $\phi_i \colon Y_i \to Y$ 
such that the images $\phi_i(Y_i)$ converge to some compact set 
$\Om \sub Y$ in the Hausdorff distance. By Theorem~\ref{thm:cptness}
we can further assume that the push-forwards
$\phi_{i\#}V_i' \in \bI_{n+1,\cs}(Y)$ converge weakly to a current
$V \in \bI_{n+1,\cs}(Y)$. Evidently $\spt(V) \sub \Om$. 
Since the (sub)sequence $(X_i)_{i \in \N}$ has asymptotic rank $\le n$,
and $\Om$ is an asymptotic subset, it follows that there is no bi-Lipschitz
embedding of a compact subset of $\R^{n+1}$ with positive Lebesgue measure
into $\Om$, and therefore $V$ must be zero
(compare Theorem~8.3 in~\cite{Lan-LC} and the comments thereafter).
Hence, the cycles $\phi_{i\#}Z_i' \in \bZ_{n,\cs}(Y)$ 
converge weakly to $\d V = 0$. We can assume that $Y$ satisfies
condition $(\CI_n)$; for example, the injective hull of $Y$ admits a conical
geodesic bicombing~\cite{Lan-IH} and is still compact.
It then follows from Theorem~\ref{thm:weak-flat} that
$\W(\phi_{i\#}Z_i') \to 0$. As this is an intrinsic property of the currents, 
we conclude that $\W(Z_i') \to 0$ with respect to the metrics
$r_i^{-1}d_i$. However, it follows from the inequality
$\FillVol_{X_i}(Z_i) \ge \eps r_i^{\,n+1}$ and~(1) that
$\FillVol_{X_i}(Z_i') \ge \eps/2$ with respect to $r_i^{-1}d_i$.
This contradicts Theorem~\ref{thm:w-fillvol}
(note that $(X_i,r_i^{-1}d_i)$ still satisfies $(\CI_n)[c]$).
\end{proof}

As a first application of Theorem~\ref{thm:sii1} we derive a 
density bound for minimizing fillings. This is similar to assertion~(5)
of Proposition~\ref{prop:approx} and to Proposition~4.5 in~\cite{KleL}.

\begin{proposition} \label{prop:min-fill}
For all\/ $C,\del > 0$ there is a $\rho = \rho(\cX,n,c,C,\del) > 0$
such that if $X$ belongs to $\cX$, and $Z \in \bZ_{n,\cs}(X)$ is a cycle with
$\D_{p,r}(Z) \le C$ for some $p \in X$ and for all $r > a \ge 0$,
then every minimizing filling $V \in \bI_{n+1,\cs}(X)$ of $Z$ satisfies
\[
  \D_{p,r}(V) = \frac{\|V\|(\B{p}{r})}{r^{n+1}} < \del
\]
for all $r > \max\{\rho,a\}$.
\end{proposition}

\begin{proof}
Let $V \in \bI_{n+1,\cs}(X)$ be a minimizing filling of $Z$, and set
$B_r := \B{p}{r}$ for all $r > 0$.
Choose a sufficiently large radius $r_0 > 0$ such that
\[
\del r_0^{\,n+1} > \M(V) \ge \|V\|(B_{r_0}),
\]
and put $r_i := 2^{-i}r_0$ for every integer $i \ge 1$.
There exists an $s \in (r_1,r_0)$ such that the slice
$T_s := \d(V \on B_s) - Z \on B_s$ is in $\bI_{n,\cs}(X)$ and has mass
$\M(T_s) \le \|V\|(B_{r_0})/(r_0 - r_1) < 2\del r_0^{\,n}$.
Furthermore, if $r_1 > a$, then $\M(Z \on B_s) \le Cs^n$ by assumption,
thus the cycle $Z_s := Z \on B_s + T_s$ satisfies
\[
\M(Z_s) \le Cs^n + 2\del r_0^{\,n} \le (C + 2\del)r_0^{\,n},
\]
and $\spt(Z_s) \sub B_{r_0}$. Note that $V \on B_s$ is a minimizing filling
of $Z_s$. By Theorem~\ref{thm:sii1} there is a constant
$\rho := 2^{-1}\rho_0(\cX,n,c,C+2\del,2^{-(n+1)}\del) > 0$ such that
if $r_1 > \max\{\rho,a\}$ and, hence, $r_0 \ge 2\rho$, then 
\[
\|V\|(B_{r_1}) \le \M(V \on B_s) < 2^{-(n+1)} \del r_0^{\,n+1} = \del r_1^{\,n+1}.
\]
Finally, given any $r > \max\{\rho,a\}$, we can choose
$r_0$ initially such that $r = r_k = 2^{-k}r_0$ for some $k \ge 1$.  
When $k \ge 2$, we repeat the above slicing argument 
successively for $i = 2,\dots,k$, with $(r_i,r_{i-1})$ in place of $(r_1,r_0)$.
This eventually shows that 
\[
\|V\|(B_r) < \del r^{n+1}
\]
for all $r > \max\{\rho,a\}$.
\end{proof}  

Next we prove a linear isoperimetric inequality for cycles
with controlled density. This yields the implication
$(\AR_n) \Rightarrow (\LII_n)$ in Theorem~\ref{thm:main}.

\begin{theorem}[linear isoperimetric inequality] \label{thm:lii}
There is a constant $\nu = \nu(n,c) > 0$, and for all $C > 0$ there is
a $\lam = \lam(\cX,n,c,C) > 0$,
such that if $X$ belongs to $\cX$ and $Z \in \bZ_{n,\cs}(X)$
is a cycle with $(C,a)$-controlled density, $a \ge 0$,
then $\FillVol_{X}(Z) \le \max\{\lam,\nu a\}\,\M(Z)$.
\end{theorem}  

\begin{proof}
Note again that for some $\gam = \gam(n,c)$, every member of $\cX$ satisfies
$(\EII_n)[\gam]$ and, if $n \ge 2$, also $(\EII_{n-1})[\gam]$.
Suppose that $Z \in \bZ_{n,\cs}(X)$ has $(C,a)$-controlled density.
For any $\eta > 0$, to be specified below, Proposition~\ref{prop:approx}
provides a $V \in \bI_{n+1,\cs}(X)$ such that,
for $Z' := Z - \d V \in \bZ_{n,\cs}(X)$
and some constants $\alpha,\theta > 0$ depending only on $n$ and $\gam$,
\begin{enumerate}
\item
  $\eta^{-1}\M(V) + \M(Z') \le \M(Z)$;
\item
  $\D_{x,r}(Z') \ge \theta$ for all $x \in \spt(Z')$ and $r \in (0,\alpha\eta]$;
\item
  $Z'$ has $(2^{n+1}C,\max\{a,2^{n+1}\eta\})$-controlled density.
\end{enumerate}
By Theorem~\ref{thm:plateau} there exists a minimizing filling
$V' \in \bI_{n+1,\cs}(X)$ of $Z'$, and there is a $\del_0 = \del_0(n,\gam) > 0$
such that $\D_{x,r}(V') \ge \del_0$ whenever $x \in \spt(V')$, $r > 0$,
and $\B{x}{r} \cap \spt(Z') = \es$. On the other hand, by~(3) and
Proposition~\ref{prop:min-fill}, there is a constant
$\rho' := \rho(\cX,n,c,2^{n+1}C,\del_0) > 0$ such that $\D_{p,r}(V') < \del_0$
for all $p \in X$ and $r > \max\{\rho',a,2^{n+1}\eta\}$.
We now fix $\eta$ so that
\[
  2^{n+1}\eta = \max\{\rho',a\}.
\]
Then $\spt(V')$ is within distance at most $2^{n+1}\eta$ from $\spt(Z')$.
Pick a maximal set $N \sub \spt(Z')$ of distinct points at mutual distance 
$> 2\alpha\eta$. The collection of all balls $\B{x}{2\alpha\eta}$ with
$x \in N$ covers $\spt(Z')$, and the corresponding balls with radius
\[
  r := 2(\alpha + 2^n)\eta
\]
cover $\spt(V')$.
Furthermore, the balls $\B{x}{\alpha\eta}$ with $x \in N$ are pairwise
disjoint, and $\|Z'\|(\B{x}{\alpha\eta}) \ge \theta (\alpha\eta)^n$ by~(2).
Hence, $|N| \le \M(Z')/(\theta\alpha^n\eta^n)$, and since $r > 2^{n+1}\eta$,
we have $\|V'\|(\B{x}{r}) < \del_0 r^{n+1}$ for all $x \in N$.
(Possibly $V' = 0$ and $|N| = 0$.) Thus
\[
  \M(V') \le |N| \,\del_0 r^{n+1}
  \le \frac{\del_0 r^{n+1}}{\theta \alpha^n\eta^n}\,\M(Z')
  \le \nu'\eta\,\M(Z')
\]
for some $\nu' = \nu'(n,c) \ge 1$.
Now $V + V'$ is a filling of $Z$ with mass
\[
  \M(V + V') \le \nu'(\M(V) + \eta\,\M(Z')) \le \nu'\eta\,\M(Z)
\]
by~(1). In view of the choice of $\eta$, this gives the result.
\end{proof}

We now turn to the sub-Euclidean isoperimetric inequality as stated in
Theorem~\ref{thm:main}.
The proof below shows that $(\LII_n) \Rightarrow (\SII_n)$.

\begin{theorem}[sub-Euclidean isoperimetric inequality]
\label{thm:sii2}  
For all $\eps > 0$ there is a constant $M_0 = M_0(\cX,n,c,\eps) > 0$
such that if $X$ belongs to $\cX$ and $Z \in \bZ_{n,\cs}(X)$,
then $\FillVol_X(Z) < \eps \max\{M_0,\M(Z)\}^{(n+1)/n}$.
\end{theorem}

\begin{proof}
Given $Z \in \bZ_{n,\cs}(X)$, note that if $t > 0$ and $r > t\,\M(Z)^{1/n}$,
then $\M(Z) < t^{-n} r^n$, thus $Z$ has $(t^{-n},t\,\M(Z)^{1/n})$-controlled
density. For $\eps > 0$, let $\nu = \nu(n,c)$ and $\lam = \lam(\cX,n,c,C)$
be the constants from Theorem~\ref{thm:lii}, where now
$C := t^{-n}$ for any fixed $t < \eps/\nu$.
Let $M_0 > 0$ be such that $\lam < \eps M_0^{1/n}$. Then
\[
  \max\{\lam,\nu t\,\M(Z)^{1/n}\}\,\M(Z) < \eps\max\{M_0,\M(Z)\}^{(n+1)/n},
\]
and the result follows from Theorem~\ref{thm:lii}.
\end{proof}  

Finally, we show that Theorem~\ref{thm:sii1} follows easily from
Theorem~\ref{thm:sii2}. 
Given $C,\eps > 0$, put $\eps' := \eps/C^{(n+1)/n}$.
If $Z \in \bZ_{n,\cs}(X)$ is a cycle with $\M(Z) \le Cr^n$, and $r$ is
sufficiently large, so that $Cr^n \ge M_0 = M_0(\eps')$, then
\[
  \FillVol_X(Z) < \eps'(Cr^n)^{(n+1)/n} = \eps r^{n+1}
\]
by Theorem~\ref{thm:sii2}. Since the asymptotic rank assumption in
Theorem~\ref{thm:lii} is only used through Theorem~\ref{thm:sii1},
this also shows that $(\SII_n) \Rightarrow (\LII_n)$.


\section{Quasiflats and quasi-minimizers} \label{sect:qmin}

A map $f \colon W \to X$ from another metric space $W$ into $X$ is an
{\em $(L,a)$-quasi-isometric embedding}, for constants $L \ge 1$ and
$a \ge 0$, if
\[
  L^{-1}d(x,y) - a \le d(f(x),f(y)) \le L\,d(x,y) + a
\]
for all $x,y \in W$. Propositions~3.6 and~3.7 in~\cite{KleL} show that 
quasi-isometric embeddings of domains $W \sub \R^n$ into $X$ give rise
to quasi-minimizing currents with controlled density,
as defined in the introduction.
An inspection of the proofs reveals that the statements hold in a
stronger form, in particular with a quasi-minimality constant $Q$ independent
of the parameter~$a$. We provide the details for convenience, and also
because parts of the proof will be used later.
The first result refers to the simpler case when the map is actually
Lipschitz.

\begin{proposition} \label{prop:lip-qflats}
For all $n \ge 1$ and $L \ge 1$ there exist\/ $C > 0$ and $Q \ge 1$ such that
the following holds. Let\/ $W \sub \R^n$ be a compact set with finite perimeter,
so that the associated current $E := \bb{W}$ (with $\spt(E) \sub W$ and
$\spt(\d E) \sub \d W$) is in $\bI_{n,\cs}(\R^n)$.
Suppose that $a \ge 0$ and $f \colon W \to X$ is an $L$-Lipschitz,
$(L,a)$-quasi-isometric embedding into a proper metric space~$X$.
Then $S := f_\#E \in \bI_{n,\cs}(X)$ has $(C,a)$-controlled density
and is $(Q,Qa)$-quasi-minimizing mod $f(\d W)$,
furthermore $d(f(x),\spt(S)) \le Qa$ for all $x \in W$ with
$d(x,\d W) > Qa$.
\end{proposition}

\begin{proof}  
Let $B := \B{p}{r}$ for some $p \in X$ and $r > a$. Then
\[
  \|f_\#E\|(B) \le L^n\,\|E\|(f^{-1}(B)),
\]
and $f^{-1}(B)$ has diameter $\le L(2r + a) \le 3Lr$, thus $\|S\|(B) \le Cr^n$
for some constant $C = C(n,L)$. This shows that $S$ has $(C,a)$-controlled
density.

Next, let $V \sub W$ be a maximal subset of distinct points at mutual
distance $> 2La$ ($V = W$ in case $a = 0$).
Note that $d(f(x),f(y)) \ge (2L)^{-1}d(x,y)$
for any $x,y \in V$, thus $f|_V$ has a $2L$-Lipschitz inverse, which we can
extend to an $\bar L$-Lipschitz map $\bar f \colon X \to \R^n$ for some 
$\bar L = \bar L(n,L)$. Put $h := \bar f \circ f \colon W \to \R^n$.
For every $x \in W$ there is a $y \in V$ with $d(x,y) \le 2La$;
then $h(y) = y$ and
\[
d(h(x),x) \le d(h(x),h(y)) + d(y,x) \le (\bar L L + 1)\,d(x,y) \le Na,
\]
where $N := 2(\bar L L + 1)L$.

Let $x \in W$. Suppose that $r > 2LNa$ and $B_r := \B{f(x)}{r}$ is disjoint
from $f(\d W)$. For almost every such $r$, both $S_r := S \on B_r$ and
$E_r := E \on f^{-1}(B_r)$ are integral currents, and $f_\#E_r = S_r$.
Since $f^{-1}(B_r) \cap \spt(\d E) = \es$, the support of $\d E_r$ lies in the
boundary of $f^{-1}(B_r)$ and is thus at distance at least $L^{-1}r$ from $x$.
The geodesic homotopy from the inclusion map $W \to \R^n$ to $h$ provides a
current $R \in \bI_{n,\cs}(\R^n)$ with $\d R = h_\#(\d E_r) - \d E_r$ such that
$\spt(R)$ is within distance $Na$ from $\spt(\d E_r)$.
In fact, $R = \bar f_\#S_r - E_r$, because
$h_\#(\d E_r) = \d(h_\# E_r) = \d(\bar f_\#S_r)$ and $\bZ_{n,\cs}(\R^n) = \{0\}$.
By the choice of $r$ we have $L^{-1}r - Na > (2L)^{-1}r$,
thus $\spt(R)$ lies outside $\B{x}{(2L)^{-1}r}$. It follows that
\[
\M(\bar f_\#S_r) = \M(E_r + R) \ge \|E\|(\B{x}{(2L)^{-1}r}) \ge \eps r^n
\]
for some $\eps = \eps(n,L) > 0$. 
Now if $T \in \bI_{n,\cs}(X)$ is such that $\d T = \d S_r$, then 
$\bar f_\#T = \bar f_\#S_r$, and
\[
\M(S_r) \le Cr^n \le C\eps^{-1} \M(\bar f_\#T) \le Q'\,\M(T)
\]
for $Q' := C\eps^{-1}\bar L^n$. This holds for all 
$x \in W$ and almost all $r > 2LNa$ as long as $\B{f(x)}{r}$ is
disjoint from $f(\d W)$. In particular, since $\spt(S) \sub f(W)$,
$S$ is $(Q',2LNa)$-quasi-minimizing mod $f(\d W)$.

Finally, put $Q := \max\{Q',L(2LN + 1)\}$. Let $x \in W$
with $d(x,\d W) > Qa$. Then $w := d(f(x),f(\d W)) > L^{-1}Qa - a \ge 2LNa$.
For almost every $r \in (2LNa,w)$, the above argument shows that 
$\M(\bar f_\#S_r) > 0$, thus $S_r = S \on \B{f(x)}{r} \ne 0$,
and this implies that $d(f(x),\spt(S)) \le 2LNa \le Qa$. 
\end{proof}

For the second result, we suppose that the compact set $W \sub \R^n$ is a
{\em triangulated\/ polyhedral set}, that is, $W$ has the structure of a
finite simplicial complex all of whose maximal cells are Euclidean
$n$-simplices.
We write $W^0$ and $(\d W)^0$ for the set of vertices and boundary vertices
of the triangulation, respectively. Furthermore,
$[\,\cdot\,]_a$ stands for the closed $a$-neighborhood of a subset of $X$.

\begin{proposition} \label{prop:qflats}
For all $n \ge 1$, $c > 0$, and $K,L \ge 1$ there exist
$C > 0$ and $Q \ge 1$ such that the following holds.
Let $X$ be a proper metric space satisfying condition~$(\CI_{n-1})[c]$.
Suppose that $a > 0$, and $W \sub \R^n$ is a triangulated polyhedral set with
simplices of diameter $\le a$ such that every ball in $\R^n$ of radius
$r > a$ intersects at most $Ka^{-n}r^n$ maximal simplices.
Let $\cP_*(W)$ denote the corresponding chain complex of
simplicial integral currents. 
If $f \colon W \to X$ is an $(L,a)$-quasi-isometric embedding,
then there exists a chain map
$\iota \colon \cP_*(W) \to \bI_{*,\cs}(X)$ such that 
\begin{enumerate}
\item
$\iota$ maps every vertex $\bb{x_0} \in \cP_0(W)$ to $\bb{f(x_0)}$ and, for 
$1 \le k \le n$, every basic oriented simplex $\bb{x_0,\dots,x_k} \in \cP_k(W)$ 
to a minimizing current with support in $[f(\{x_0,\dots,x_k\})]_{Qa}$; 
\item 
$S := \iota\bb{W} \in \bI_{n,\cs}(X)$ has $(C,a)$-controlled density and
is $(Q,Qa)$-quasi-minimizing mod $[f((\d W)^0)]_{Qa}$;
\item 
$d(f(x),\spt(S)) \le Qa$ for all $x \in W$ with $d(x,(\d W)^0) > Qa$. 
\end{enumerate}
\end{proposition}

Note that by~(1), $\spt(S) \sub [f(W^0)]_{Qa}$ and 
$\spt(\d S) \sub [f((\d W)^0)]_{Qa}$.

\begin{proof}
Put $\cS_* := \bigcup_{k=0}^n \cS_k$, where $\cS_k$ denotes the set of all 
basic oriented simplices $s = \bb{x_0,\dots,x_k} \in \cP_k(W)$
(compare p.~365 in~\cite{Fed} for the notation).
We define a map $\iota \colon \cS_* \to \bI_{*,\cs}(X)$ by induction on $k$.
For $\bb{x_0} \in \cS_0$, we put $\iota\bb{x_0} := f_\#\bb{x_0} = \bb{f(x_0)}$. 
Suppose now that $k \in \{1,\ldots,n\}$ and $\iota$ is defined on $\cS_{k-1}$.
For every $k$-cell of $W$, we choose an orientation
$s = \bb{x_0,\dots,x_k} \in \cS_k$, then we let 
$\iota(s) \in \bI_{k,\cs}(X)$ be a minimizing filling of the cycle
\[
\sum_{i=0}^k (-1)^i\,\iota\bb{x_0,\dots,x_{i-1},x_{i+1},\dots,x_k}
\in \bZ_{k-1,\cs}(X),
\]
and we put $\iota(-s) := -\iota(s)$. 
The resulting map on $\cS_*$ readily extends to a chain map 
$\iota \colon \cP_*(W) \to \bI_{*,\cs}(X)$.
Note that $f$ maps the vertex set of any cell of $W$ to a set of diameter
at most $(L+1)a$. It follows inductively from condition~$(\CI_{n-1})[c]$ and
Theorem~\ref{thm:plateau} (if $n \ge 2$, then $X$ satisfies~$(\EII_{n-1})$
by Theorem~\ref{thm:eii}) that for all $s = \bb{x_0,\dots,x_k} \in \cS_k$,
\[
\M(\iota(s)) \le Ma^k
\]
and $\spt(\iota(s)) \sub [f(\{x_0,\dots,x_k\})]_{Ma}$ for some 
constant $M \ge L + 1$ depending only on $n,c,L$.

Let now $\cS_n^+ \sub \cS_n$ be the set of all positively oriented 
$n$-simplices, whose sum is $\bb{W}$. Put $S := \iota\bb{W}$. 
To show that $S$ has controlled density, 
let $p \in X$ and $r > a$, and consider the set of all $s \in \cS_n^+$ for 
which $\spt(\iota(s)) \cap \B{p}{r} \ne \es$. Every such $s$ has a 
vertex $x^s$ with $f(x^s) \in \B{p}{r + Ma}$, thus the set of all
$x^s$ has diameter at most $L(2(r + Ma) + a) \le L(2M + 3)r$.
It follows that there are at most $Ka^{-n}(L(2M + 3)r)^n$ such simplices
and that
\[
  \D_{p,r}(S) \le C := KL^n(2M + 3)^nM
\]
for $p \in X$ and $r > a$.  

Similarly as in the proof of Proposition~\ref{prop:lip-qflats}, 
there exists an $\bar L$-Lipschitz map $\bar f \colon X \to \R^n$ such that 
$h := \bar f \circ f \colon W \to \R^n$ satisfies
\[
  d(h(x),x) \le Na
\]
for all $x \in W$,
where $\bar L$ and $N$ depend only on $n,L$. Then
\[
\bar\iota := \bar f_\# \circ \iota \colon \cP_*(W) \to \bI_{*,\cs}(\R^n)
\]
is a chain map that sends every $\bb{x_0} \in \cS_0$ to $\bb{h(x_0)}$ and
every $\bb{x_0,\dots,x_k} \in \cS_k$ to a current with support in
$[\{x_0,\dots,x_k\}]_{(\bar L M + N)a}$. 
Let $\cP_*(\d W)$ be the complex of simplicial integral currents in $\d W$.
A similar inductive construction as above, using minimizing fillings of
cycles in $\R^n$, produces a chain homotopy between the inclusion map 
$\cP_*(\d W) \to \bI_{*,\cs}(\R^n)$ and the restriction of $\bar\iota$ to
$\cP_*(\d W)$. This yields an $R \in \bI_{n,\cs}(\R^n)$ with boundary
$\d R = \bar\iota(\d\bb{W}) - \d\bb{W}$
and support $\spt(R) \sub [(\d W)^0]_{\bar Ma}$ for some constant
$\bar M = \bar M(n,c,L) \ge 1$. In fact,
\[
  R = \bar f_\#S - \bb{W},
\]
because $\bar\iota(\d\bb{W}) = \d(\bar\iota\bb{W}) = \d(\bar f_\#S)$.

Note that $\spt(S) \sub [f(W^0)]_{Ma}$. Let $\bar x \in [f(W^0)]_{Ma}$ and
$r > 0$ be such that $\B{\bar x}{r} \cap f((\d W)^0) = \es$ and
$S_r := S \on \B{\bar x}{r} \in \bI_{n,\cs}(X)$. 
We want to show that if $r > Pa$, for some sufficiently large constant
$P = P(n,c,L) \ge 1$, then $\M(\bar f_\# S_r) \ge \eps r^n$
for some $\eps = \eps(n,L) > 0$. Choose $x \in  W^0$ with
$d(f(x),\bar x) \le Ma$, and put $B_x := \B{x}{(2L)^{-1}r}$.
For all $y \in (\d W)^0$,
\[
  r < d(\bar x,f(y)) \le d(f(x),f(y)) + Ma \le L\,d(x,y) + (M+1)a
\]
and thus $d(x,y) > (2L)^{-1}r + \bar Ma$ for sufficiently large $P$;
then
\[
  (\spt(R) \cup \d W) \cap B_x = \es.
\]
Moreover, for every $\bar y \in \spt(S - S_r) \sub \spt(S)$ there is a vertex
$y \in W^0$ such that $d(f(y),\bar y) \le Ma$,
\[
  r \le d(\bar x,\bar y) \le d(f(x),f(y)) + 2Ma \le L\,d(x,y) + (2M+1)a,
\]
and $d(x,y) \le d(x,\bar f(\bar y)) + d(\bar f(\bar y),h(y)) + Na
\le d(x,\bar f(\bar y)) + (\bar L M + N)a$;  
thus $d(x,\bar f(\bar y)) > (2L)^{-1}r$ for sufficiently large $P$,
implying that
\[
\spt(\bar f_\#(S - S_r)) \cap B_x = \es.
\]
Since $\bar f_\#S_r = \bb{W} + R - \bar f_\#(S-S_r)$, it then follows that 
\[
\M(\bar f_\#S_r) \ge \|\bb{W}\|(B_x) \ge \eps r^n
\]
for some $\eps = \eps(n,L) > 0$, as desired.
Now if $T \in \bI_{n,\cs}(X)$ is such that $\d T = \d S_r$, then 
$\bar f_\#T = \bar f_\#S_r$, and
\[
\M(S_r) \le Cr^n \le C\eps^{-1} \M(\bar f_\#T) \le Q'\,\M(T)
\]
for $Q' := C\eps^{-1}\bar L^n$. Since $\spt(S)$ and $\spt(\d S)$ are within
distance $Ma$ from $f(W^0)$ and $f((\d W)^0)$, respectively, this shows
in particular that $S$ is $(Q',Pa)$-quasi-minimizing mod $[f((\d W)^0)]_{Ma}$.

Finally, put $Q := \max\{M,Q',L(P + 1)\}$. Let $x' \in W$ with
$d(x',(\d W)^0) > Qa$. Then $w := d(f(x'),f((\d W)^0)) > L^{-1}Qa - a \ge Pa$.
Note that $f(x') \in [f(W^0)]_{Ma}$, as $M \ge L + 1$.
For $\bar x = f(x')$ and almost every $r \in (Pa,w)$, the above argument
shows that $\M(\bar f_\#S_r) > 0$, thus $S_r = S \on \B{\bar x}{r} \ne 0$,
and this implies that $d(f(x'),\spt(S)) \le Pa \le Qa$. 
\end{proof}


\section{Morse lemma, slim simplices, and filling radius}
\label{sect:slim}

We now turn to the remaining assertions in Theorem~\ref{thm:main}.
For the first three results, we assume as in Sect.~\ref{sect:ii}
that $\cX$ is a class of proper metric spaces such that for some $n \ge 1$
and $c > 0$, all members of $\cX$ satisfy condition $(\CI_n)[c]$,
and every sequence $(X_i)_{i \in \N}$ in $\cX$ has asymptotic rank $\le n$.
We begin with a uniform version of the Morse lemma, analogous to
Theorem~5.1 in~\cite{KleL}. The asymptotic rank assumption is only used
through Proposition~\ref{prop:min-fill}, or Theorem~\ref{thm:sii1}, which
in turn follows from Theorem~\ref{thm:sii2}.
Hence, $(\SII_n) \Rightarrow (\ML_n)$.

\begin{theorem}[Morse lemma] \label{thm:morse}
For all $C > 0$ and $Q \ge 1$ there is a constant $l = l(\cX,n,c,C,Q) \ge 0$
such that if $X$ belongs to $\cX$, and $Z \in \bZ_{n,\cs}(X)$
has $(C,a)$-controlled density and is $(Q,a)$-quasi-minimizing mod $Y$,
where $Y \sub X$ is a closed set and $a \ge 0$,
then the support of $Z$ is within distance at most $\max\{l,4a\}$ from~$Y$.
\end{theorem}  

\begin{proof}  
Let $x \in \spt(Z) \sm Y$.  
Essentially the same argument as for the second part of
Theorem~\ref{thm:plateau} (using $(\EII_{n-1})$ if $n \ge 2$) shows
that there is a constant $\del_0' = \del_0'(n,c) > 0$ such that 
$\D_{x,s}(Z) \ge \del_0'\,Q^{1-n}$ whenever $s > 2a$ and
$\B{x}{s} \cap Y = \es$ (see Lemma~3.3 in~\cite{KleL}).
Now let $V \in \bI_{n+1,\cs}(X)$ be a
minimizing filling of $Z$, and suppose that $r > 4a$ and
$\B{x}{r} \cap Y  = \es$. For almost every $s \in (2a,r)$, the slice
$T_s = \d(V \on \B{x}{s}) - Z \on \B{x}{s} \in \bI_{n,\cs}(X)$ satisfies
\[
  Q\,\M(T_s) \ge \M(Z \on \B{x}{s}) \ge \del_0'\,Q^{1-n}s^n,
\]
and integrating the inequality $\M(T_s) \ge \del_0'\,Q^{-n}s^n$ from
$r/2 > 2a$ to $r$ we get that $\D_{x,r}(V) \ge \del$ for some
$\del = \del(n,c,Q) > 0$ (compare Lemma~3.4 in~\cite{KleL}). 
On the other hand, by Proposition~\ref{prop:min-fill}
there is a constant $l := \rho(\cX,n,c,C,\del) > 0$ such that
$\D_{x,r}(V) < \del$ for all $r > \max\{l,a\}$.
Hence, $r \le \max\{l,4a\}$. 
\end{proof}  

The next statement strengthens Theorem~5.2 in~\cite{KleL}.
The proof shows that $(\ML_n) \Rightarrow (\SS_n)$.
A {\em facet\/} of an $(n+1)$-simplex is an $n$-dimensional face.

\begin{theorem}[slim simplices] \label{thm:slim}
For all $L \ge 1$ there is a constant $D = D(\cX,n,c,L) \ge 0$ such that
the following holds. Let $\Del$ be a Euclidean $(n+1)$-simplex, $X$ a member
of $\cX$, and $a \ge 0$. Suppose that $f \colon \d\Del \to X$ is a map whose
restriction to each facet of\/ $\Del$ is an $(L,a)$-quasi-isometric embedding.
Then the image of every facet is within distance at most $D(1+a)$ from the
union of the images of the remaining ones.
\end{theorem}

\begin{proof}
Let $W_0,\dots,W_{n+1} \sub \d\Del$ be an enumeration of the (closed) facets
of $\Del$, and let $E_i := (\d\bb{\Del}) \on W_i \in \bI_{n,\cs}(\R^{n+1})$
denote the corresponding currents, whose sum is the boundary cycle
$\d\bb{\Del} \in \bZ_{n,\cs}(\R^{n+1})$.

Suppose that $a > 0$. 
Choose a triangulation of $\d\Del$ with simplices of diameter $\le a$ such
that, for some constant $K = K(n)$ and for each $i$, any ball in $\R^{n+1}$
of radius $r > a$ intersects at most $K a^{-n}r^n$ maximal simplices in $W_i$.
Let $\cP_*(\d\Del)$ be the corresponding chain complex of simplicial integral
currents. A slight adaptation of Proposition~\ref{prop:qflats}
provides a chain map $\iota \colon \cP_*(\d\Del) \to \bI_{*,\cs}(X)$
such that the following properties hold for each 
$S_i := \iota(E_i) \in \bI_{n,\cs}(X)$ and for some constants $C,Q$ 
depending only on $n,c,L$:
\begin{enumerate}
\item
$\spt(S_i) \sub [f(W_i)]_{Qa}$ and $\spt(\d S_i) \sub [f(\d W_i)]_{Qa}$;
\item
$S_i$ has $(C,a)$-controlled density and is $(Q,Qa)$-quasi-minimizing
mod $[f(\d W_i)]_{Qa}$;
\item
$d(f(x),\spt(S_i)) \le Qa$ for all $x \in W_i$ with $d(x,\d W_i) > Qa$.
\end{enumerate}
Here $[\,\cdot\,]_{Qa}$ stands again for the closed $Qa$-neighborhood, and
$\d W_i$ denotes the relative boundary of $W_i$.
Let $M_i$ denote the union of all $W_j$ with $j \ne i$.
The cycle $Z := \iota(\d\bb{\Del}) = \sum_{i=0}^{n+1} S_i$ has
$((n+2)C,a)$-controlled density and is
$(Q,Qa)$-quasi-minimizing mod~$[f(M_i)]_{Qa}$ for every $i$.
It then follows from Theorem~\ref{thm:morse}
that the set $\spt(S_i) \sm [f(M_i))]_{Qa} = \spt(Z) \sm [f(M_i)]_{Qa}$ is
within distance at most $\max\{l',4Qa\}$ from $[f(M_i)]_{Qa}$ for some
$l' = l'(\cX,n,c,L)$.
Hence, for any $x \in W_i$, it follows from~(3) that
$d(f(x),f(M_i))$ is less than or equal to $2Qa + \max\{l',4Qa\}$ 
if $d(x,\d W_i) > Qa$ and less than or equal to $LQa + a$ otherwise.

Note that if the restriction of $f$ to each facet of $\Del$ is
$L$-Lipschitz in addition, or if $a = 0$, then the proof can be simplified
by using Proposition~\ref{prop:lip-qflats} instead of
Proposition~\ref{prop:qflats}. 
\end{proof}

The proof of the following result relies again on
Proposition~\ref{prop:min-fill}; thus $(\SII_n) \Rightarrow (\FR_n)$.

\begin{theorem}[filling radius] \label{thm:fill-rad}
For all $C > 0$ there is a constant $h = h(\cX,n,c,C) > 0$ such that
if $X$ belongs to $\cX$ and $Z \in \bZ_{n,\cs}(X)$ has $(C,a)$-controlled
density for some $a \ge 0$, then the support of every minimizing filling
$V \in \bI_{n+1,\cs}(X)$ of $Z$ is within distance at most
$\max\{h,a\}$ from $\spt(Z)$.
\end{theorem}

\begin{proof}
Suppose that $x \in \spt(V) \sm \spt(Z)$.
By Theorem~\ref{thm:eii} and Theorem~\ref{thm:plateau} there are constants
$\gam = \gam(n,c)$ and $\del_0 = \del_0(n,\gam) > 0$ such that
$\D_{x,r}(V) \ge \del_0$ whenever $r > 0$ and $\B{x}{r} \cap \spt(Z) = \es$.
On the other hand, Proposition~\ref{prop:min-fill} shows that there
is a constant $h = \rho(\cX,n,c,C,\del_0) > 0$ such that
$\D_{x,r}(V) < \del_0$ for all $r > \max\{h,a\}$. Thus there is no
point $x \in \spt(V)$ at distance bigger than $\max\{h,a\}$ from
$\spt(Z)$.
\end{proof}

We now prove the implication $(\SS_n) \Rightarrow (\AR_n)$, which
holds without further assumptions on the metric space $X$.

\begin{proposition} \label{prop:ss-ar}
Let $(X_i)_{i \in \N}$ be a sequence of metric spaces $X_i = (X_i,d_i)$,
let $n \ge 1$, and suppose that for every $L \ge 1$ there exists $D \ge 0$
such that every $X_i$ satisfies $(\SS_n)$ with constant $D = D(L)$.
Then the sequence $(X_i)_{i \in \N}$ has asymptotic rank $\le n$.
\end{proposition}

\begin{proof}
Suppose to the contrary that the sequence $(X_i)_{i \in \N}$ has asymptotic
rank $> n$. Then there exist a compact set $K \sub \R^{n+1}$ with positive
Lebesgue measure, an $L$-bi-Lipschitz map $\phi \colon K \to \Om$
onto some metric space $\Om$, and a sequence of $(1,\del_i)$-quasi-isometric
embeddings $h_i \colon \Om \to (X_i,r_i^{-1}d_i)$, where $L \ge 1$,
$\del_i \to 0$, and $r_i \to \infty$.
We can assume that $0 \in \R^{n+1}$ is a Lebesgue density point of $K$.
Let $B := \B{0}{1} \sub \R^{n+1}$.
For all $k \in \N$ there is a $\lam_k > 0$ such that every point
in $\lam_k B$ is at distance $\le (2k)^{-1}\lam_k$ from some point in $K$,
thus there exist $(1,k^{-1}\lam_k)$-quasi-isometric embeddings
$\psi_k \colon \lam_kB \to K$.
Choose $i(k) \in \N$ such that $s_k := \lam_k r_{i(k)} \to \infty$
and $\eps_k := \lam_k^{-1}\del_{i(k)} + k^{-1}L \to 0$.
It is straightforward to check that the map
\[
  f_k \colon s_kB \to (X_{i(k)},d_{i(k)})
\]
defined by $f_k(s_kx) = h_{i(k)} \circ \phi \circ \psi_k(\lam_kx)$
for all $x \in B$ is an $(L,\eps_ks_k)$-quasi-isometric embedding.

Now let $\Del$ be any $(n+1)$-simplex inscribed in $B$, and pick a point
$x$ in a facet of $\Del$ such that $x$ is a distance
$\del > 0$ away from the union $M$ of the remaining facets.
For every $k$, the point $f_k(s_kx)$ is at distance at least
$L^{-1}\del s_k - \eps_ks_k$ from $f_k(s_kM)$. On the other hand, by
assumption, there is a constant $D = D(L)$ such that the distance is at
most $D(1 + \eps_ks_k)$. This leads to the inequality
$L^{-1}\del - \eps_k \le D(s_k^{-1} + \eps_k)$, which contradicts
the fact that $s_k \to \infty$ and $\eps_k \to 0$.
\end{proof}

Lastly, we show that $(\FR_n) \Rightarrow (\AR_n)$. This is similar
to Theorem~6.1 in~\cite{Wen-AR}.

\begin{proposition} \label{prop:fr-ar}
Let $n \ge 1$ and $c > 0$, and let $(X_i)_{i \in \N}$ be a sequence of proper
metric spaces $X_i = (X_i,d_i)$ satisfying $(\CI_n)[c]$. Suppose further that
for every $C > 0$ there exists $h > 0$ such that every $X_i$ satisfies 
$(\FR_n)$ with constant $h = h(C)$. Then the sequence $(X_i)_{i \in \N}$
has asymptotic rank $\le n$.
\end{proposition}

\begin{proof}
Suppose to the contrary that $(X_i)_{i \in \N}$ has asymptotic rank $> n$.
Let $L,s_k,\eps_k$ and $f_k \colon s_kB \to X_{i(k)}$
be given as in the first part of the proof of Proposition~\ref{prop:ss-ar}.
Let again $\Del$ be any $(n+1)$-simplex inscribed in $B$.
For every $k$, fix a triangulation of $\d\Del$ with simplices of
diameter at most $\eps_k$ such that, for some constant $K = K(n)$,
every ball in $\R^{n+1}$ of radius $r > \eps_k$ meets at most
$K(r/\eps_k)^n$ maximal simplices in each facet of $\Del$.
It then follows as in the proof of Proposition~\ref{prop:qflats}
that for every $k$, and for some constants $C,\bar L,\bar M$ depending
only on $n,c,L$, there exist a cycle $Z_k \in \bZ_{n,\cs}(X_{i(k)})$ with
$(C,\eps_ks_k)$-controlled density,
an $\bar L$-Lipschitz map $\bar f_k \colon X_{i(k)} \to \R^{n+1}$,
and a current $R_k \in \bI_{n+1,\cs}(\R^{n+1})$ such that
$\d R_k = \bar f_{k\#}Z_k - \d\bb{s_k\Del}$ and
\[
  \spt(R_k) \cup \bar f_k(\spt(Z_k)) \sub [\d(s_k\Del)]_{\bar M\eps_k s_k}.
\]  
Fix a point $x \in \Del$ a distance $\del > 0$ away from $\d\Del$.
Suppose that $k$ is so large that $\bar M\eps_k < \del$,
and $V_k \in \bI_{n+1,\cs}(X_{i(k)})$ is any filling of $Z_k$.
Then $\d\bb{s_k\Del} = \d(\bar f_{k\#}V_k) - \d R_k$, hence
$\bb{s_k\Del} = \bar f_{k\#}V_k - R_k$ and
$s_k\Del \sub \spt(\bar f_{k\#}V_k) \cup \spt(R_k)$.
Since $s_kx \not\in \spt(R_k)$,
there is a point $y_k \in \spt(V_k)$ such that $\bar f_k(y_k) = s_kx$.
For every $z_k \in \spt(Z_k)$, we have
\[
  s_k\del = d(s_kx,s_k\Del) \le d(s_kx,\bar f_k(z_k)) + \bar M\eps_k s_k
  \le \bar L\,d_{i(k)}(y_k,z_k) + \bar M\eps_k s_k.
\]
On the other hand, by assumption, there exists a filling $V_k$ of $Z_k$
whose support is within distance $\max\{h,\eps_ks_k\}$ from $\spt(Z_k)$,
where $h = h(C)$. This leads to the inequality
$\del \le \bar L \max\{s_k^{-1}h,\eps_k\} + \bar M \eps_k$,
which contradicts the fact that $s_k \to \infty$ and $\eps_k \to 0$.
\end{proof}

The uniform statements in Sect.~\ref{sect:ii} and above can be combined
to show that the implications in Theorem~\ref{thm:main} that we proved
through $(\AR_n)$ hold with constants independent of $X$.
We exemplify this for $(\FR_n) \Rightarrow (\SII_n)$.
If $\cX$ denotes the class of all proper metric spaces satisfying
$(\CI_n)[c]$ and $(\FR_n)$ for some fixed $c > 0$ and $h = h(C)$,
then Proposition~\ref{prop:fr-ar} shows that every sequence $(X_i)_{i \in \N}$
in $\cX$ has asymptotic rank $\le n$. Hence, by Theorem~\ref{thm:sii2},
$(\SII_n)$ holds for some constant $M_0 = M_0(\cX,n,c,\eps) > 0$,
which depends only on $n,c,\eps$ and the function $h = h(C)$.

\subsection*{Acknowledgement}
We thank Stefan Wenger for a useful discussion. 



\end{document}